\documentclass[twoside,english,american,DIV12,listtotoc,bibtotoc,idxtotoc]{scrartcl}
\usepackage[latin1]{inputenc}
\usepackage{babel}
\usepackage{verbatim}
\usepackage{amsthm}
\usepackage{amsmath}
\usepackage{amssymb}
\usepackage[unicode=true,pdfusetitle,
 bookmarks=true,bookmarksnumbered=false,bookmarksopen=false,
 breaklinks=false,pdfborder={0 0 1},backref=false,colorlinks=false]
 {hyperref}

\makeatletter
\theoremstyle{plain}
\newtheorem{thm}{\protect\theoremname}[section]
  \theoremstyle{definition}
  \newtheorem{defn}[thm]{\protect\definitionname}
  \theoremstyle{remark}
  \newtheorem{rem}[thm]{\protect\remarkname}
  \theoremstyle{definition}
  \newtheorem{example}[thm]{\protect\examplename}
  \theoremstyle{plain}
  \newtheorem{prop}[thm]{\protect\propositionname}
  \theoremstyle{plain}
  \newtheorem{cor}[thm]{\protect\corollaryname}
  \theoremstyle{plain}
  \newtheorem{lem}[thm]{\protect\lemmaname}

\usepackage{helvet}
\usepackage[T1]{fontenc}

\usepackage{amsmath}
\usepackage{setspace}
\usepackage{amssymb}
\usepackage{enumerate}
\usepackage{url}

\makeatletter

\usepackage[T1]{fontenc}    % Silbentrennung auch nach Umlauten m?glich

\usepackage{a4wide}        % A4-Format
\usepackage{tu-preprint}
\usepackage{amsmath}

\usepackage{euscript}
\usepackage{mathtools}
\allowdisplaybreaks[1] % allow page breaks in some displayed equations
\usepackage{amsfonts}
\newenvironment{keywords}{ \noindent\footnotesize\textbf{Keywords and phrases:}}{}

\newenvironment{class}{\noindent\footnotesize\textbf{Mathematics subject classification 2010:}}{}

\usepackage{color,tu-preprint}

\newcommand*{\dive}{\operatorname{div}}

\newcommand*{\curl}{\operatorname{curl}}

\newcommand*{\grad}{\operatorname{grad}}

\newcommand*{\supp}{\operatorname{supp}}

\renewcommand*{\i}{\mathrm{i}}
\DeclareMathAccent{\Circ}{\mathalpha}{operators}{"17}
\newcommand{\interior}[1]{\Circ{#1}}

\renewcommand{\Re}{\operatorname{\mathfrak{Re}}}

\newcommand{\oi}[2]{\left]#1,#2 \right[}
\newcommand{\rga}[1]{\left]#1,\infty  \right[}
\newcommand{\rla}[1]{\left]-\infty ,#1 \right[}

\newcommand{\rlc}[1]{\left]-\infty ,#1 \right]}

\renewcommand{\tilde}{\widetilde}
\renewcommand*{\epsilon}{\varepsilon}
\renewcommand*{\rho}{\varrho}

\makeatother

\usepackage{babel}
\institut{Institut f\"ur Analysis}

\preprintnumber{MATH-AN-011-2012}

\preprinttitle{On a Comprehensive Class of Linear Control Problems.}

\author{Rainer Picard, Sascha Trostorff \& Marcus Waurick}

\AtBeginDocument{
  
}

\makeatother

  \addto\captionsamerican{\renewcommand{\corollaryname}{Corollary}}
  \addto\captionsamerican{\renewcommand{\definitionname}{Definition}}
  \addto\captionsamerican{\renewcommand{\examplename}{Example}}
  \addto\captionsamerican{\renewcommand{\lemmaname}{Lemma}}
  \addto\captionsamerican{\renewcommand{\propositionname}{Proposition}}
  \addto\captionsamerican{\renewcommand{\remarkname}{Remark}}
  \addto\captionsamerican{\renewcommand{\theoremname}{Theorem}}
  \addto\captionsenglish{\renewcommand{\corollaryname}{Corollary}}
  \addto\captionsenglish{\renewcommand{\definitionname}{Definition}}
  \addto\captionsenglish{\renewcommand{\examplename}{Example}}
  \addto\captionsenglish{\renewcommand{\lemmaname}{Lemma}}
  \addto\captionsenglish{\renewcommand{\propositionname}{Proposition}}
  \addto\captionsenglish{\renewcommand{\remarkname}{Remark}}
  \addto\captionsenglish{\renewcommand{\theoremname}{Theorem}}
  \providecommand{\corollaryname}{Corollary}
  \providecommand{\definitionname}{Definition}
  \providecommand{\examplename}{Example}
  \providecommand{\lemmaname}{Lemma}
  \providecommand{\propositionname}{Proposition}
  \providecommand{\remarkname}{Remark}
\providecommand{\theoremname}{Theorem}

\begin{document}
\makepreprinttitlepage

\author{ Rainer Picard, \\ Sascha Trostorff, \\ Marcus Waurick \\ Institut f\"ur Analysis, Fachrichtung Mathematik\\ Technische Universit\"at Dresden\\ Germany\\ rainer.picard@tu-dresden.de\\ sascha.trostorff@tu-dresden.de\\ marcus.waurick@tu-dresden.de }

\title{On a Comprehensive Class of Linear Control Problems.}

\maketitle
\begin{abstract} \textbf{Abstract.} We discuss a class of linear
control problems in a Hilbert space setting. This class encompasses
such diverse systems as port-Hamiltonian systems, Maxwell's equations
with boundary control or the acoustic equations with boundary control
and boundary observation. The boundary control and observation acts
on abstract boundary data spaces such that the only geometric constraint
on the underlying domain stems from requiring a closed range constraint
for the spatial operator part, a requirement which for the wave equation
amounts to the validity of a Poincare-Wirtinger-type inequality. We
also address the issue of conservativity of the control problems under
consideration. \end{abstract}


\begin{class} 93C05 (Linear systems), 93C20 (Systems governed by
partial differential equations), 93C25 (Systems in abstract spaces)
 \end{class}

\newpage

\tableofcontents{} 

\newpage

\section{Introduction}

Finite-dimensional linear control problems are commonly discussed
in the form of a differential-algebraic system. The first system equation
links the \emph{state }$x$ taking values in $\mathbb{R}^{n}$ to
the \emph{control} or \emph{input }$u$, which takes values in $\mathbb{R}^{m}$
via matrices $A,B,\mu_{0}$ of appropriate size in the way 
\[
\mu_{0}\dot{x}(t)=Ax(t)+Bu(t),\quad t\in\rga0.
\]

If $\mu_{0}$ is boundedly invertible, the latter equation is also
known as \emph{state differential equation}, in general we could have
here a \emph{state differential-algebraic equation}. This equation
is completed by some initial condition for the part of the state variable
that gets differentiated, i.e. $(\mu_{0}x)(0+)=\mu_{0}x_{0}$. In
control theory one is mainly interested in the \emph{observation}
or \emph{output} $y$, which is a $\mathbb{R}^{l}$-valued function
given by the \emph{observation equation}
\[
y(t)=Cx(t)+Du(t)\quad t\in\rga0,
\]

for suitable matrices $C$ and $D$. 

Thus, denoting the time-derivative by $\partial_{0}$ and using the
whole real line $\mathbb{R}$ instead of $\rga0$, which transforms
the initial condition into a Dirac-$\delta$-source term on the right-hand
side, we arrive at the following system
\begin{equation}
\left(\begin{array}{cc}
\partial_{0}\mu_{0}-A & 0\\
-C & 1
\end{array}\right)\left(\begin{array}{c}
x\\
y
\end{array}\right)=\left(\begin{array}{c}
B\\
D
\end{array}\right)u+\left(\begin{array}{c}
\delta\otimes\mu_{0}x_{0}\\
0
\end{array}\right).\label{eq:system1}
\end{equation}
Here for time-continuous states $\Phi$ we have $\left(\delta\otimes x_{0}\right)\Phi\coloneqq x_{0}^{*}\Phi\left(0\right)$. 

In essence, with the added observation equation we are just considering
a larger differential-algebraic equation with an implied specific
block structure. 

Making $x,u$ the unknowns and treating $y$ as a term on the right-hand
side we arrive at the alternative formulation
\begin{equation}
\left(\begin{array}{cc}
A-\partial_{0}\mu_{0} & B\\
C & D
\end{array}\right)\left(\begin{array}{c}
x\\
u
\end{array}\right)=\left(\left(\begin{array}{cc}
A & B\\
C & D
\end{array}\right)-\partial_{0}\left(\begin{array}{cc}
\mu_{0} & 0\\
0 & 0
\end{array}\right)\right)\left(\begin{array}{c}
x\\
u
\end{array}\right)=\left(\begin{array}{c}
0\\
1
\end{array}\right)y-\left(\begin{array}{c}
\delta\otimes\mu_{0}x_{0}\\
0
\end{array}\right).\label{eq:system2}
\end{equation}
Whereas well-posedness issues are discussed in connection with respect
to (\ref{eq:system1}) (given control $u$, unknown output $y$) the
-- in a sense -- inverse problem (\ref{eq:system2}) (given output
$y$, unknown control $u$) is the usual starting point of discussion
of control system leading in the commonly discussed case $\mu_{0}=1$
to the analysis of $2\times2$ block matrices $\left(\begin{array}{cc}
A & B\\
C & D
\end{array}\right)$.

Systems of such general block structure have been generalized to the
infinite-dimensional case. In this case $A,B,C$ and $D$ are linear
operators in suitable Hilbert spaces. A solution theory for this problem
is rather straightforward, if one assumes that $\mu_{0}=1$ and $A$
is a generator of a strongly continuous semi-group and the operators
$B,C$ and $D$ are bounded linear operators. 

If one studies systems with boundary control, the assumption on $B$
and $C$ to be bounded has to be lifted. Hence, more sophisticated
techniques need to be used to establish well-posedness of such systems
even if $\mu_{0}=1$ is assumed, \cite{Salamon1987,Salamon1989,Curtain1989,Weiss1989-2,Engel1998,LasTrigg2000_1,LasTrigg2000_2,Weiss2003,Jacob2004}.
In the light of the rather sophisticated considerations required to
deal with such a situation the question arises if a different perspective
may shed some new insight on this problem class. Taking our guidance
from the discussion in a book by Lasiecka and Triggiani \cite{LasTrigg2000_1}
and two seminal papers by Tucsnak and Weiss \cite{Weiss2003,WeissStaffTucs},
where a class of systems is specified by $\left(\begin{array}{cc}
A & B\\
C & D
\end{array}\right)$ with $A$ being a semi-group generator and $B,C$ operators, which
are not bounded operators between state and control space is considered,
it has been found, \cite{PiTroWau_2012}, that by introducing an additional
state variable we get an equivalent system with a different $2\times2$-block
structure $\left(\begin{array}{cc}
A & B\\
C & D
\end{array}\right)$, where now $A$ is even skew-selfadjoint%
\footnote{For two operators $A,B$ defined on a Hilbert space, we say that $A$
is the adjoint of $B$ if $A=B^{*}$, and we say that $A$ is selfadjoint
if $A=A^{*}.$ In order to consistently extend this terminology to
the case, when $A=-B^{*}$, we choose to say that $A$ is the skew-adjoint
of $B$. Therefore, if $A=-A^{*}$, we say that $A$ is skew-selfadjoint.%
} and $B,C,D$ are all bounded linear operators. However, since $\mu_{0}$
is not invertible the (semi-)group for $A$ is of little help to obtain
well-posedness. Fortunately, there is a whole machinery to attack
differential-algebraic systems directly without resorting to one-parameter
semi-group techniques. The solution strategy relies solely on the
fact that -- in a suitable Hilbert space setting -- the whole differential-algebraic
system operator together with its adjoint is strictly positive definite.
Since -- by elementary Hilbert space functional analysis -- strict
positive definiteness of a closed operator $T$ and of its adjoint
$T^{\ast}$ implies that $0$ is an element in the resolvent sets
of both operators, it would probably be difficult to find a more basic
well-posedness class than this one. Surprisingly, however, this class
is spacious enough to cover all classical linear evolution problems
of mathematical physics and allows for convenient generalization to
more complex ``material relations''. The solution concept does not
require the existence of a fundamental solution. Therefore questions
naturally arising in the semi-group context such as whether an operator
is admissible or not (\cite{Engel1998,Jacob2004,LasTrigg2000_1,LasTrigg2000_2,Salamon1987,Weiss1989-2})
can be by-passed and replaced by a mere regularization requirement
rather than the well-posedness of the respective equations.

In this note, we shall present a unified way of looking at control
problems of this type as differential-algebraic systems, which may
make the solution theory more easily accessible. More precisely, we
will provide evidence that linear control problems can readily be
understood as evolutionary equations, a particular class of differential-algebraic
equations, which have been studied and used for many applications
to other fields, see \cite{PDE_DeGruyter}. We will show that a large
class of linear (boundary) control systems fits into this class. We
exemplify these observation with linear boundary control problems
studied by \cite{Weiss1989,Weiss2003,LasTrigg2000_1,LasTrigg2000_2,Weck2000}.
It should be noted, however, that the class presented here is much
larger, since we are not limited to cases, where one-parameter semi-group
strategies can successfully be utilized. This having been said, it
also has to be admitted that the results of this paper are merely
addressing the foundation of control problems. Actual control issues
such as controllability, reachability, stability etc. are beyond the
scope of this paper and may constitute future research.

In the process of developing our framework for boundary control systems
we shall also make a particular effort at developing a theoretical
setting for dealing with arbitrary boundaries of underlying domains,
which is of importance in more realistic applications, where boundary
smoothness is not reasonable to assume. This way we are saved from
using boundary trace results, which are hard to come by or unavailable
for example for domains with cuts, cusps, line segments or fractal
boundaries. However, the general well-posedness results are independent
of this theoretical setting, which in any case may also be substituted
by more classical boundary trace ideas, if requiring sufficient smoothness
of the boundary is not an issue.

A particular subclass of port-Hamiltonian systems (\cite{JacZwart1,GorZwarMa2006,ZwartGor2010})
can be discussed within this theory. As a by-product we give a possible
generalization of boundary control systems similar to port-Hamiltonian
systems to the case of more than one spatial dimension, which appeared
to be, at least to the best of the authors' knowledge, an open problem. 

We will also address the issue of conservativity. In fact, we show
a certain type of impedance conservativity \cite{Ball_Staffans,Malinen_Staffans_impedance,Malinen_Staffans_JDE,Staffans_passive,Staffasn_Jen,WeissStaffTucs,Weiss2003}.
Thereby, we show that the hypotheses on the structure of the material
law in \cite{PiTroWau_2012} can be weakened. We obtain a certain
general energy-balance equality, imposing assumptions on the structure
of the equation that are easily verified in applications. 

In Section \ref{sec:Functional-analytic-framework}, we give the functional
analytic preliminaries needed to discuss evolutionary equations in
the sense of \cite{Pi2009-1}. This includes the time-derivative realized
as a normal, continuously invertible operator and the notion of Sobolev-chains.

Section \ref{sec:Control-Systems-as} states the notion of abstract
linear control systems defined as a subclass of particular evolutionary
systems. We show well-posedness of the respective systems under easily
verifiable conditions on the structure of the operators involved.
In essence, this section recalls the well-posedness theorem of \cite{Pi2009-1}
including the notion of causality defined in \cite{Kalauch}. 

Section \ref{sec:Conservative-Systems} discusses the qualitative
property of conservativity for abstract linear control systems. In
order to show conservativity of abstract linear control systems, a
particular structure of the operators involved and a regularizing
property of the solution operator associated to the system is needed.
The regularizing property is slightly stronger than the one in \cite{PiTroWau_2012}.
As a trade-off, the structural requirements on the operators involved
are less restrictive. 

The subsequent section, Section \ref{sec:Boundary-control}, provides
a way to embed linear boundary control systems into abstract linear
control systems. For an account on boundary control systems dealt
with in the literature, we refer the reader to \cite{Avrov,Malinen_Staffans_impedance,Malinen_Staffans_JDE,Salamon1987,Weck2000,Weiss1989-2,Weiss2003,ZwartGor2010},
where also strategies from the theory of selfadjoint extensions of
symmetric operators come into play, \cite{Behrndt_bdy_rel,Behrndt_Kreusler_bdy_rel,Derkach_bdy_relations,SWIPs,BDY_systems}.
As a first illustrative example of boundary control systems we discuss
in Subsection \ref{sub:Port-Hamiltonian-systems} the notion of port-Hamiltonian
systems as introduced in \cite{JacZwart1}, also see \cite{JacZwartIsem}.
In order to give higher-dimensional analogues for a particular subclass
of port-Hamiltonian systems, we define abstract boundary data spaces
(Subsection \ref{sub:Boundary-Data-Spaces}). The latter can and will
be introduced in a purely operator-theoretic framework. Consequently,
in applications these spaces may be defined without any regularity
assumptions on the underlying domain. The main idea is to replace
the classical trace spaces, which may not be defined in the general
situation of irregular boundaries, with an abstract analogue of ``1-harmonic
functions''. Subsection \ref{sub:Boundary-Control-and} provides
the solution theory of a class of abstract linear control systems
with boundary control and boundary observation.

The last section, Section \ref{sec:Applications}, is devoted to illustrate
our previous findings. We give an alternative way to show the well-posedness
of Maxwell's equation with boundary control similar to the one discussed
in \cite{Weck2000} (Subsection \ref{sub:Maxwell's-Equation-with})
and the well-posedness of a wave equation with boundary control and
observation generalizing the one discussed in \cite{Weiss2003} (Subsection
\ref{sub:Boundary-control-and-wave}).

\section{Functional-Analytic Framework\label{sec:Functional-analytic-framework}}

In this section we introduce the framework for evolutionary equations,
which will be defined in the next section. The relevant statements
of the results can be found in more detail in \cite{PDE_DeGruyter}.
First, following \cite{Kalauch}, we define the time-derivative as
a normal, boundedly invertible operator in a suitable $L^{2}$-type
space:
\begin{defn}
For $\nu\in\rga0$ we denote by $H_{\nu,0}(\mathbb{R})$ the space
of all square-integrable functions%
\footnote{Throughout we identify the equivalence classes induced by the equality
almost everywhere with their representatives.%
} with respect to the exponentially weighted Lebesgue-measure $\exp(-2\nu t)\mbox{ d}t$,
equipped with the inner product given by 
\[
\langle f|g\rangle_{H_{\nu,0}(\mathbb{R})}\coloneqq\intop_{\mathbb{\mathbb{R}}}f(t)^{\ast}g(t)\exp(-2\nu t)\mbox{ d}t\quad(f,g\in H_{\nu,0}(\mathbb{R})).
\]
\end{defn}
\begin{rem}
From the definition of $H_{\nu,0}(\mathbb{R})$ we see that the operator
$\exp(-\nu m):H_{\nu,0}(\mathbb{R})\to L^{2}(\mathbb{R}),$ defined
by $\left(\exp(-\nu m)f\right)(t)=\exp(-\nu t)f(t)$, $t\in\mathbb{R}$,
is unitary. Furthermore, it is clear that the space $\interior C_{\infty}(\mathbb{R})$,
the space of indefinitely differentiable functions with compact support
on $\mathbb{R}$, is dense in $H_{\nu,0}(\mathbb{R}).$ \end{rem}
\begin{defn}
Let $\nu>0$. We denote by%
\footnote{For the space of $L^{2}$-functions defined on an open subset $\Omega\subseteq\mathbb{R}^{n}$
with distributional gradient lying in $L^{2}(\Omega)^{n}$ we use
the notation $H^{1}(\Omega)$. If the gradient is only locally square-integrable,
we write $H_{\text{loc}}^{1}(\Omega)$.%
} $\partial:H^{1}(\mathbb{R})\subseteq L^{2}(\mathbb{R})\to L^{2}(\mathbb{R})$
the usual weak derivative on $L^{2}(\mathbb{R}),$ which is known
to be skew-selfadjoint, i.e., $\partial^{*}=-\partial$. We set 
\[
\partial_{0,\nu}\coloneqq\exp(-\nu m)^{-1}(\partial+\nu)\exp(-\nu m)
\]
as the derivative operator on $H_{\nu,0}(\mathbb{R}).$ For convenience
we will write $\partial_{0}$ instead of $\partial_{0,\nu}$ if the
particular choice of $\nu>0$ is clear from the context. \end{defn}
\begin{rem}
The operator $\partial_{0,\nu}$ is normal with $\Re\partial_{0,\nu}=\nu$.
Moreover, since the operator $\exp(-\nu m)^{-1}\partial\exp(-\nu m)$
is skew-selfadjoint, we get that $0\in\rho(\partial_{0,\nu})$ and
$\|\partial_{0,\nu}^{-1}\|\leq\frac{1}{\nu}.$ To justify our choice
of $\partial_{0,\nu}$ as the derivative we compute $\partial_{0,\nu}\phi$
for $\phi\in\interior C_{\infty}(\mathbb{R})$:
\begin{align*}
\left(\partial_{0,\nu}\phi\right)(t) & =\exp(\nu t)\left((\partial+\nu)\exp(-\nu m)\phi\right)(t)\\
 & =\exp(\nu t)(-\nu\exp(-\nu m)\phi+\exp(-\nu m)\phi'+\nu\exp(-\nu m)\phi)(t)\\
 & =\phi'(t)
\end{align*}
for all $t\in\mathbb{R}.$ 
\end{rem}
Next we need the (standard) concept of so-called Sobolev-chains or
rigged Hilbert spaces. The proofs of the following assertions can
be found, for instance, in \cite[Chapter 2]{PDE_DeGruyter}.
\begin{defn}
Let $H$ be a Hilbert space and $C:D(C)\subseteq H\to H$ be a densely
defined, closed linear operator with $0\in\rho(C).$ For $k\in\mathbb{Z}$
we set $H_{k}(C)$ as the completion of the domain $D(C^{k})$ with
respect to the norm $|C^{k}\cdot|_{H}.$ Then $\left(H_{k}(C)\right)_{k\in\mathbb{Z}}$
becomes a sequence of Hilbert spaces such that $H_{k}(C)$ is continuously
and densely embedded into $H_{k-1}(C)$ for each $k\in\mathbb{Z}.$
We call $(H_{k}(C))_{k\in\mathbb{Z}}$ the \emph{Sobolev-chain of
$C$. }We define 
\begin{align*}
H_{\infty}(C) & \coloneqq\bigcap_{k\in\mathbb{Z}}H_{k}(C),\\
H_{-\infty}(C) & \coloneqq\bigcup_{k\in\mathbb{Z}}H_{k}(C).
\end{align*}
\end{defn}
\begin{rem}
For $k\in\mathbb{N}\setminus\{0\}$ the operator 
\begin{align*}
C:H_{k}(C) & \to H_{k-1}(C)\\
x & \mapsto Cx
\end{align*}

\end{rem}
is unitary. For $-k\in\mathbb{N}$ consider the operator 
\begin{align*}
C:H_{\infty}(C)\subseteq H_{k}(C) & \to H_{k-1}(C)\\
x & \mapsto Cx.
\end{align*}

This operator turns out to be densely defined, isometric with dense
range, hence it can be extended to a unitary operator (again denoted
by $C$) $C:H_{k}(C)\to H_{k-1}(C).$ 

$ $
\begin{rem}
~

(a) The Hilbert space $H_{k}(C)$ for $k\in\mathbb{Z}$ can be identified
with the dual space $H_{-k}(C^{\ast})^{\ast}$ using the following
unitary mapping 
\begin{align*}
U:H_{k}(C) & \to H_{-k}(C^{\ast})^{\ast}\\
x & \mapsto\left(y\mapsto\langle C^{k}x|\left(C^{\ast}\right)^{-k}y\rangle_{H}\right).
\end{align*}
This allows an extension of the inner product $\langle\cdot\,|\,\cdot\rangle$
in $H$ to a continuous sesqui-linearform 
\begin{align*}
\langle\cdot\,|\,\cdot\rangle:H_{k}(C)\times H_{-k}(C^{\ast}) & \to\mathbb{C}
\end{align*}
in the sense of the dual pairing $\left(H_{k}(C),H_{-k}(C^{\ast})\right).$
We will not distinguish between the inner product given on $H$ and
its extension to such pairings.

(b) Let $U$ be a Hilbert space and $A:H_{1}(C)\to U$ be a linear
bounded operator. Then the dual operator $A':U^{\ast}\to H_{1}(C)^{\ast}$
can be identified with the operator $A^{\diamond}:U\to H_{-1}(C^{\ast}),$
by identifying the dual space $U^{\ast}$ with $U$ and the space
$H_{1}(C)^{\ast}$ with $H_{-1}(C^{\ast})$ according to the aforementioned
unitary mapping.\end{rem}
\begin{example}
Choosing $H=H_{\nu,0}(\mathbb{R})$ for some $\nu>0$ and $C=\partial_{0}$
we can construct the Sobolev-chain associated to $\partial_{0}$.
We will use the notation $H_{\nu,k}(\mathbb{R})\coloneqq H_{k}(\partial_{0})$
for $k\in\mathbb{Z}.$ The Dirac-distribution $\delta$ is an element
of $H_{\nu,-1}(\mathbb{R})$ and $\partial_{0}^{-1}\delta=\chi_{\rga0}.$\end{example}
\begin{rem}
For a densely defined closed linear operator $A:D(A)\subseteq H_{0}\to H_{1}$,
where $H_{0}$ and $H_{1}$ are two Hilbert spaces, we can construct
the Sobolev-chain to $|A|+\i$ and $|A^{\ast}|+\i$, respectively.
Then $A$ and $A^{\ast}$ can be established as bounded linear operators
\[
A:H_{k}(|A|+\i)\to H_{k-1}(|A^{\ast}|+\i)
\]
and
\[
A^{\ast}:H_{k}(|A^{\ast}|+\i)\to H_{k-1}(|A|+\i)
\]
for all $k\in\mathbb{Z}.$ 
\end{rem}
Not only the concept of Sobolev-chains is of use in the later sections
but also the one of Sobolev-lattices. A possible way to define them
is with the help of tensor product constructions. For the theory of
tensor products see e.g. \cite{Weidmann} and for the concept of Sobolev-lattices
we refer the reader to \cite[Chapter 2]{PDE_DeGruyter}.
\begin{rem}
Let $\nu>0$ and $H$ a Hilbert space. For a densely defined closed
linear operator $C:D(C)\subseteq H\to H$ with $0\in\rho(C)$ we consider
the canonical extension $1_{H_{\nu,0}(\mathbb{R})}\otimes C$ of $C$
to the space $H_{\nu,0}(\mathbb{R})\otimes H$, where $1_{H_{\nu,0}(\mathbb{R})}$
denotes the identity on $H_{\nu,0}(\mathbb{R}).$ Analogously we extend
$\partial_{0}$ to the space $H_{\nu,0}(\mathbb{R})\otimes H$ by
taking the tensor product $\partial_{0}\otimes1_{H}$ with the identity
$1_{H}$ on $H$. We re-use the notation $C$ and $\partial_{0}$
for their respective extensions to the space $H_{\nu,0}(\mathbb{R})\otimes H.$
Then the operators $\partial_{0}$ and $C$ can be established as
operators on $H_{\nu,-\infty}(\mathbb{R})\otimes H_{-\infty}(C)\coloneqq\bigcup_{k,j\in\mathbb{Z}}H_{\nu,k}(\mathbb{R})\otimes H_{j}(C).$
More precisely, 
\[
\partial_{0}:H_{\nu,k}(\mathbb{R})\otimes H_{j}(C)\to H_{\nu,k-1}(\mathbb{R})\otimes H_{j}(C)
\]
and 
\[
C:H_{\nu,k}(\mathbb{R})\otimes H_{j}(C)\to H_{\nu,k}(\mathbb{R})\otimes H_{j-1}(C)
\]
are unitary operators for each $k,j\in\mathbb{Z}.$ As a matter of
convenience, we will also write $H_{\nu,k}(\mathbb{R},H)$ for all
$k\in\mathbb{Z}\cup\left\{ -\infty,\infty\right\} $ for $H_{\nu,k}(\mathbb{R})\otimes H$
(or $\cup_{l\in\mathbb{Z}}H_{\nu,l}(\mathbb{R})\otimes H$ or $\cap_{l\in\mathbb{Z}}H_{\nu,l}(\mathbb{R})\otimes H$)
to stress the unitary equivalence of the tensor products of these
Hilbert spaces with the respective space of (generalized) Hilbert-space-valued
functions.
\end{rem}

\section{Control Systems as Special Evolutionary Problems\label{sec:Control-Systems-as} }

In Section \ref{sec:Boundary-control}, we shall show that many linear
control systems fit into the following particular class.
\begin{defn}
Let $H,V$ be Hilbert spaces, $M_{0},M_{1}\in L(H),$ $J\in L(V,H)$
and $A:D(A)\subseteq H\to H$ skew-selfadjoint. For $\nu\in\rga{0}$,
we define the set 
\[
\mathcal{E}_{M_{0},M_{1},A,J}^{\nu}\coloneqq\left\{ (x,f)\in H_{\nu,-\infty}(\mathbb{R},H\oplus V)|(\partial_{0}M_{0}+M_{1}+A)x=Jf\right\} .
\]
The set $\mathcal{E}_{M_{0},M_{1},A,J}\coloneqq\bigcup_{\nu>0}\mathcal{E}_{M_{0},M_{1},A,J}^{\nu}$
is called \emph{evolutionary system}. The system $\mathcal{E}_{M_{0},M_{1},A,J}$
is called \emph{well-posed} if there exists $\nu_{0}\in\rga{0}$ such
that for all $\nu\in[\nu_{0},\infty[$ the relation 
\[
\mathcal{S}_{M_{0},M_{1},A,J}^{\nu}\coloneqq\{(f,x)|(x,f)\in\mathcal{E}_{M_{0},M_{1},A,J}^{\nu}\cap H_{\nu,0}(\mathbb{R},H\oplus V)\}\subseteq H_{\nu,0}(\mathbb{R},V)\oplus H_{\nu,0}(\mathbb{R},H)
\]
defines a densely defined, continuous linear mapping from $H_{\nu,0}(\mathbb{R},V)$
to $H_{\nu,0}(\mathbb{R},H)$. We call $\mathcal{S}_{M_{0},M_{1},A,J}^{\nu}$
\emph{solution operator (for $\nu$)}. \end{defn}
\begin{thm}[\cite{PiTroWau_2012,Pi2009-1}]
\label{Thm: sol_th_ev_syst}Let $\mathcal{E}_{M_{0},M_{1},A,J}$
be an evolutionary system. Assume that $M_{0}=M_{0}^{*}$ and that
there exists $c\in\rga{0}$ such that
\[
\nu M_{0}+\Re M_{1}\geq c>0
\]
for all sufficiently large $\nu\in\rga{0}$. Then $\mathcal{E}_{M_{0},M_{1},A,J}$
is well-posed and the corresponding solution operator\emph{ }$\mathcal{S}_{M_{0},M_{1},A,J}^{\nu}$
is \emph{causal}, i.e., for all $a\in\mathbb{R}$ we have 
\[
\chi_{\rlc a}(m_{0})\mathcal{S}_{M_{0},M_{1},A,J}^{\nu}\chi_{\rlc a}(m_{0})=\chi_{\rlc a}(m_{0})\mathcal{S}_{M_{0},M_{1},A,J}^{\nu},
\]
where $\chi_{\rlc a}(m_{0})$ denotes the operator of multiplying
with the cut-off function $\chi_{\rlc a}.$ 
\end{thm}
The following proposition can be found in \cite{PiTroWau_2012}. The
basic fact, which is used in the proof is that $\partial_{0,\nu}^{-1}$
commutes with $\mathcal{S}_{M_{0},M_{1},A,J}^{\text{\ensuremath{\nu}}}$
for a well-posed evolutionary system $\mathcal{E}_{M_{0},M_{1},A,J}^{\text{ }}$,
for all sufficiently large $\nu\in\rga{0}$.
\begin{prop}
\label{Prop: distributional_time}Let $\mathcal{E}_{M_{0},M_{1},A,J}$
be a well-posed evolutionary system. Then, for all sufficiently large
$\nu\in\rga{0}$, we have that $\mathcal{S}_{M_{0},M_{1},A,J}^{\text{\ensuremath{\nu}}}$
uniquely extends to a continuous linear operator from $H_{\nu,k}(\mathbb{R},V)$
to $H_{\nu,k}(\mathbb{R},H)$ for all $k\in\mathbb{Z}$. \end{prop}
\begin{rem}
This proposition provides a way to model initial value problems, since
initial conditions can be represented as a Dirac-$\delta$-source
term, which turns out to be an element of the space $H_{\nu,-1}(\mathbb{R},H).$
\end{rem}
We can now describe abstract linear control systems as particular
evolutionary systems.
\begin{defn}
An evolutionary system $\mathcal{E}_{M_{0},M_{1},A,J}^{\text{ }}$
is called \emph{abstract linear control system} if there exist Hilbert
spaces $H_{0},H_{1},Y,U_{1},$ a densely defined, closed linear operator
$F:D(F)\subseteq H_{0}\to H_{1}$, $B\in L(U_{1},H)$ such that $H=H_{0}\oplus H_{1}\oplus Y$,
$A=\left(\begin{array}{ccc}
0 & -F^{*} & 0\\
F & 0 & 0\\
0 & 0 & 0
\end{array}\right)$, $V=H\oplus U_{1},$ and $J=\left(\begin{array}{cc}
1 & B\end{array}\right)$. The Hilbert spaces $H_{0}\oplus H_{1}$, $U_{1}$ and $Y$ are called
\emph{state, control }and \emph{observation space, }respectively\emph{.}
We also write $\mathcal{C}_{M_{0},M_{1},F,B}$ to denote an abstract
linear control system. \end{defn}
\begin{cor}
\label{cor:control_well_posed}Let $\mathcal{C}_{M_{0},M_{1},F,B}$
be an abstract linear control system. Assume that $M_{0}$ is selfadjoint
and that 
\[
\nu M_{0}+\Re M_{1}\geq c>0
\]
holds for all sufficiently large $\nu\in\rga{0}$. Then $\mathcal{C}_{M_{0},M_{1},F,B}$
is well-posed and the corresponding solution operators are causal.
The solution operators uniquely extend to continuous linear operators
from $H_{\nu,k}(\mathbb{R},H\oplus U_{1})$ to $H_{\nu,k}(\mathbb{R},H)$
for all $k\in\mathbb{Z}$ and $\nu\in\rga{0}$ sufficiently large.\end{cor}
\begin{proof}
Observing that $\left(\begin{array}{ccc}
0 & -F^{*} & 0\\
F & 0 & 0\\
0 & 0 & 0
\end{array}\right)$ is a skew-selfadjoint operator, we are in the situation of Theorem
\ref{Thm: sol_th_ev_syst} and Proposition \ref{Prop: distributional_time}. 
\end{proof}

\section{Conservative Systems\label{sec:Conservative-Systems}}

In this section, we consider a qualitative property of solutions to
particular linear evolutionary equations, namely that of conservativity.
For this, a suitable regularizing property has to be additionally
imposed. As a slightly modified version to the definition given in
\cite{PiTroWau_2012}, we define (locally) regularizing systems as
follows:
\begin{defn}
Let $\mathcal{E}_{M_{0},M_{1},A,J}$ be a well-posed evolutionary
system. We say that $\mathcal{E}_{M_{0},M_{1},A,J}$ is \emph{(locally)
regularizing} if the following conditions are satisfied

\begin{enumerate}[(a)]

\item  There exists $U\subseteq D(A)$ dense in $H$ such that for
all $T\in\mathbb{R}$ and $\nu\in\rga{0}$ sufficiently large 
\[
\chi_{\rla T}(m_{0})P_{0}\left((\partial_{0}M_{0}+M_{1}+A)^{-1}\delta\otimes M_{0}-\chi_{\rga0}\otimes P_{0}\right)[U]\subseteq\chi_{\rla T}(m_{0})[H_{\nu,1}(\mathbb{R},H)],
\]
where $P_{0}:H\to H$ denotes the orthogonal projector onto $M_{0}[H],$
the range of $M_{0}.$ 

\item  There exists $C\in\rga0$ such that for all $\Phi\in H$ we
have for all $T\in\mathbb{R}$ and $\nu\in\rga{0}$ sufficiently large
\[
\chi_{\rla T}(m_{0})\left(\partial_{0,\nu}M_{0}+M_{1}+A\right)^{-1}(\delta\otimes M_{0}\Phi)\in H_{\nu,0}(\mathbb{R},H)
\]
and \foreignlanguage{english}{
\[
\left|\chi_{\rla{T}}\left(m_{0}\right)\:\left(\left(\partial_{0}M_{0}+M_{1}+\mathcal{A}\right)^{-1}\delta\otimes M_{0}\Phi\right)\right|_{H_{\nu,0}(\mathbb{R};H)}\leq C\left|\Phi\right|_{H}.
\]
}

\end{enumerate}\end{defn}
\begin{rem}
As we shall see in our discussion of regularizing evolutionary systems,
it often suffices to study the following weaker norm on the left-hand
side of the estimate in (b): $|f|_{\epsilon,\nu,-1,1}:=\sup_{\phi\in H_{\nu,1}(-\epsilon,\epsilon;H),|\phi|\leqq1}|\langle\phi,f\rangle_{H_{\nu,0}(\mathbb{R},H)}|+|\chi_{\rga{\epsilon}}(m)f|_{H_{\nu,0}(\mathbb{R},H)}.$
Then the modified inequality to impose is: for all $T\in\mathbb{R}$
and $\nu\in\rga{0}$ sufficiently large and all $\epsilon\in\rga{0}$
there exists $C\in\rga{0}$ such that 
\[
\left|\chi_{_{\rla{T}}}\left(m_{0}\right)\:\left(\left(\partial_{0}M_{0}+M_{1}+\mathcal{A}\right)^{-1}\delta\otimes M_{0}\Phi\right)\right|_{\epsilon,\nu,-1,1}\leq C\left|\Phi\right|_{H}.
\]

\end{rem}
We first will consider a conservation property for evolutionary systems.
In the light of \cite{Weiss2003} this can be interpreted as a energy
balance equality. In fact we will see later on that this balance equality
may be interpreted as impedance conservativity, see e.g.~\cite{Ball_Staffans}
and also \cite{Malinen_Staffans_impedance,Malinen_Staffans_JDE,Staffans_passive,Staffasn_Jen,WeissStaffTucs}.
\begin{thm}
\label{thm:conservative_evo}Let $\mathcal{E}_{M_{0},M_{1},A,J}$
be a regularizing well-posed evolutionary system. Let $u_{0}\in H$
and consider the solution $x\in H_{\nu,-1}(\mathbb{R},H)$ of the
equation
\[
(\partial_{0}M_{0}+M_{1}+A)x=\delta\otimes M_{0}u_{0}.
\]

Then the following \emph{conservation equation} holds%
\footnote{Note that $\chi_{\rla T}(m_{0})x\in H_{\nu,0}(\mathbb{R},H)$ for
each $T\in\mathbb{R}$ according to the second assumption for regularizing
systems.%
} 
\[
\int_{[a,b]}\langle x|\Re M_{1}x\rangle_{H}=\frac{1}{2}\langle x|M_{0}x\rangle_{H}(a)-\frac{1}{2}\langle x|M_{0}x\rangle_{H}(b)
\]
for almost every $a,b\in\rga{0}$ with $b>a$.\end{thm}
\begin{proof}
Let $v_{0}\in U$. Since $\mathcal{E}_{M_{0},M_{1},A,J}$ is well-posed
there is a solution $y\in H_{\nu,-1}(\mathbb{R},H)$ of 
\[
(\partial_{0}M_{0}+M_{1}+A)y=\delta\otimes M_{0}v_{0}.
\]
This can be re-written as
\begin{equation}
\partial_{0}M_{0}(y-\chi_{\rga{0}}\otimes v_{0})+M_{1}(y-\chi_{\rga{0}}\otimes v_{0})+A(y-\chi_{\rga{0}}\otimes v_{0})=-\chi_{\rga{0}}\otimes M_{1}v_{0}-\chi_{\rga{0}}\otimes Av_{0}\label{eq:ivp_2}
\end{equation}
from which we read off that $y-\chi_{\rga{0}}\otimes v_{0}\in H_{\nu,0}(\mathbb{R},H)$
and hence $y\in H_{\nu,0}(\mathbb{R},H).$ Let $\phi\in\interior C_{\infty}(\rga{0})$
and set $T\coloneqq\sup\supp\phi.$ By assumption we have that $\chi_{\rla T}(m_{0})P_{0}(y-\chi_{\rga{0}}\otimes v_{0})\in\chi_{\rla T}(m_{0})[H_{\nu,1}(\mathbb{R},H)]$
and hence we get from (\ref{eq:ivp_2}) that 
\[
y-\chi_{\rga{0}}\otimes v_{0}\in\chi_{\rla T}(m_{0})[H_{\nu,0}(\mathbb{R},H_{1}(A+1))].
\]
Since $v_{0}\in U\subseteq D(A)$ we obtain that $y\in\chi_{\rla T}(m_{0})[H_{\nu,0}(\mathbb{R},H_{1}(A+1))]$.
We apply $\Re\langle\phi y|\cdot\rangle_{H_{\nu,0}(\mathbb{R},H)}$
to (\ref{eq:ivp_2}) and obtain
\[
\Re\langle\phi y|\partial_{0}M_{0}(y-\chi_{\rga{0}}\otimes v_{0})\rangle_{H_{\nu,0}(\mathbb{R},H)}+\Re\langle\phi y|M_{1}y\rangle_{H_{\nu,0}(\mathbb{R},H)}+\Re\langle\phi y|Ay\rangle_{H_{\nu,0}(\mathbb{R},H)}=0.
\]
Since $y$ takes values in the domain of $A$ and since $A$ is skew-selfadjoint,
we get
\begin{equation}
\Re\langle\phi y|\partial_{0}M_{0}(y-\chi_{\rga{0}}\otimes v_{0})\rangle_{H_{\nu,0}(\mathbb{R},H)}+\Re\langle\phi y|M_{1}y\rangle_{H_{\nu,0}(\mathbb{R},H)}=0.\label{eq:pre_cons_1}
\end{equation}
Since this holds for every $\phi\in\interior C_{\infty}(\rga{0})$
it follows that 
\begin{equation}
\Re\langle y|\partial_{0}M_{0}(y-\chi_{\rga{0}}\otimes v_{0})\rangle_{H}=-\Re\langle y|M_{1}y\rangle_{H}\mbox{ a.e. on }\rga{0}.\label{eq:pre_cons_2}
\end{equation}
Let $a,b\in\rga{0}$ with $a<b.$ From $\chi_{\rla b}(m_{0})P_{0}(y-\chi_{\rga{0}}\otimes v_{0})\in\chi_{\rla b}(m_{0})[H_{\nu,1}(\mathbb{R},H)]$
we get that $\left(P_{0}y\right)^{\prime}\in L^{2}(]a,b[,H)$ with
\[
(P_{0}y)'=\partial_{0}P_{0}(y-\chi_{\rga{0}}\otimes v_{0})\mbox{ on }]a,b[
\]
and thus, integrating equation (\ref{eq:pre_cons_2}) over $[a,b]$
gives 
\[
\frac{1}{2}\langle y|M_{0}y\rangle_{H}(a)=\int_{[a,b]}\langle y|\Re M_{1}y\rangle_{H}+\frac{1}{2}\langle y|M_{0}y\rangle_{H}(b).
\]
Let now $u_{0}\in H$ and $(v_{n})_{n}$ a sequence in $U$ converging
to $u_{0}$ in $H$. For $n\in\mathbb{N}$ let $y_{n}\coloneqq(\partial_{0}M_{0}+M_{1}+A)^{-1}\delta\otimes M_{0}v_{n}$
and $x\coloneqq(\partial_{0}M_{0}+M_{1}+A)^{-1}\delta\otimes M_{0}u_{0}$.
Then for every $T\in\mathbb{R}$ we can estimate:
\begin{align*}
|\chi_{\rlc T}(x-y_{n})|_{H_{\nu,0}(\mathbb{R},H)} & =|\chi_{\rlc T}(\partial_{0}M_{0}+M_{1}+A)^{-1}(\delta\otimes M_{0}u_{0}-\delta\otimes M_{0}v_{n})|_{H_{\nu,0}(\mathbb{R},H)}\\
 & \le C|u_{0}-v_{n}|_{H}
\end{align*}
where $C$ is chosen according to assumption (b) for regularizing
systems. As $n\to\infty$ we may assume $y_{n}\to x$ almost everywhere
on $\rlc b$ by re-using the notation for a suitable subsequence of
$(y_{n})_{n\in\mathbb{N}}$ and consequently $\int_{[a,b]}\langle y_{n}|\Re M_{1}y_{n}\rangle_{H}\to\int_{[a,b]}\langle x|\Re M_{1}x\rangle_{H}$
for all $a,b\in\mathbb{R}.$ Thus, the conservation equation for $x$
holds almost everywhere.
\end{proof}

\subsection*{On the Structure of Conservative Control Systems}

For the particular case of an abstract linear control systems, we
shall derive now a different conservation property based on our observation
concerning evolutionary systems. Following the block structure of
the operator matrix $A$ for the operators $M_{0}$ and $M_{1}$ we
shall denote the corresponding entries of $M_{0}$ and $M_{1}$ as
$M_{0,ij}$ and $M_{1,ij}$ respectively for $i,j\in\{0,1,2\}.$ Analogously
we may write the operator $B\in L(U_{1},H)=L(U_{1},H_{0}\oplus H_{1}\oplus Y)$
as a row vector $(B_{0}\: B_{1}\: B_{2})$, where $B_{i}\in L(U_{1},H_{i})$
for $i\in\{0,1\}$ and $B_{2}\in L(U_{1},Y).$
\begin{thm}
\label{thm:conservative_control}Let $\mathcal{C}_{M_{0},M_{1},F,B}$
be an abstract linear control system. Assume that $M_{0}$ is selfadjoint
and that there exists $c>0$ such that for all $\nu>0$ large enough,
we have $\nu M_{0}+\Re M_{1}\ge c$. Moreover, assume that $\mathcal{C}_{M_{0},M_{1},F,B}$
is a locally regularizing evolutionary system and that $M_{0,20}=0,$
$M_{0,21}=0$, $M_{0,22}=0$. Assume the \emph{compatibility conditions}%
\footnote{Note that the condition $\nu M_{0}+\Re M_{1}\ge c$ together with
$M_{0,20}=0,$ $M_{0,21}=0$, $M_{0,22}=0$ implies that $M_{1,22}$
is continuously invertible.%
}
\begin{align*}
\left(M_{1,22}^{-1}M_{1,20}\right)^{*}B_{2} & =B_{0}\text{{\,\ and\,}}\left(M_{1,22}^{-1}M_{1,21}\right)^{*}B_{2}=B_{1}.
\end{align*}

Then for $(v,w,y)\in H_{\nu,-1}(\mathbb{R},H_{0}\oplus H_{1}\oplus Y)$
and $u\in H_{\nu,0}(\mathbb{R},U)$ satisfying 
\begin{align*}
\left(\partial_{0}M_{0}+M_{1}+\left(\begin{array}{ccc}
0 & -F^{*} & 0\\
F & 0 & 0\\
0 & 0 & 0
\end{array}\right)\right)\left(\begin{array}{c}
v\\
w\\
y
\end{array}\right) & =\delta\otimes M_{0}\left(\begin{array}{c}
v_{0}\\
w_{0}\\
y_{0}
\end{array}\right)+Bu
\end{align*}
for some $(v_{0},w_{0},y_{0})\in H_{0}\oplus H_{1}\oplus Y$ the \emph{control
conservation equation }holds:\emph{
\begin{align*}
\frac{1}{2}\left\langle \left.\left(\begin{array}{c}
v\\
w\\
y
\end{array}\right)\right|M_{0}\left(\begin{array}{c}
v\\
w\\
y
\end{array}\right)\right\rangle _{H}(a)-\frac{1}{2}\left\langle \left.\left(\begin{array}{c}
v\\
w\\
y
\end{array}\right)\right|M_{0}\left(\begin{array}{c}
v\\
w\\
y
\end{array}\right)\right\rangle _{H}(b)=\\
\int_{[a,b]}\left(\left\langle \left.\left(\begin{array}{c}
v\\
w\\
y
\end{array}\right)\right|\Re M_{1}\left(\begin{array}{c}
v\\
w\\
y
\end{array}\right)\right\rangle _{H}-\left\langle \left.B_{2}u\right|\Re M_{1,22}^{-1}B_{2}u\right\rangle _{Y}\right)
\end{align*}
}for a.e. $a,b\in\rga0$ with $a<b$.
\end{thm}
Before we come to the proof of Theorem \ref{thm:conservative_control},
we will discuss an easy example. More precisely, we discuss a connection
to the so-called impedance conservativity in the sense of \cite{Ball_Staffans},
where the focus is on realization theory. 
\begin{example}
If we let $-\tilde{A}=\left(\begin{array}{cc}
0 & -F^{*}\\
F & 0
\end{array}\right)$, $M_{0}=\left(\begin{array}{ccc}
1 & 0 & 0\\
0 & 1 & 0\\
0 & 0 & 0
\end{array}\right),$ $M_{1}=\left(\begin{array}{ccc}
0 & 0 & 0\\
0 & 0 & 0\\
-C_{0} & -C_{1} & 1
\end{array}\right)$, $B=\left(\begin{array}{c}
B_{0}\\
B_{1}\\
D
\end{array}\right)$ for suitable (bounded) operators $B_{0},B_{1},C_{0},C_{1},D$. Abbreviating
$x=\left(\begin{array}{c}
v\\
w
\end{array}\right)$, $C=\left(\begin{array}{cc}
C_{0} & C_{1}\end{array}\right)$ and $\tilde{B}=\left(\begin{array}{c}
B_{0}\\
B_{1}
\end{array}\right)$, we may rewrite the equation%
\footnote{For simplicity, we assume zero initial conditions.%
} $\left(\partial_{0}M_{0}+M_{1}+\left(\begin{array}{cc}
-\tilde{A} & 0\\
0 & 0
\end{array}\right)\right)\left(\begin{array}{c}
x\\
y
\end{array}\right)=\left(\begin{array}{c}
\tilde{B}\\
D
\end{array}\right)u$ as
\[
\left(\begin{array}{c}
\partial_{0}x\\
y
\end{array}\right)=\left(\begin{array}{cc}
\tilde{A} & \tilde{B}\\
C & D
\end{array}\right)\left(\begin{array}{c}
x\\
u
\end{array}\right).
\]
Note that in this particular situation the block structure of $A$
corresponds to the one of $M_{0}$, which we did not assume in Theorem
\ref{thm:conservative_control}. However, in this particular case,
we may compare the asserted conservativity in Theorem \ref{thm:conservative_control}
with the conservative realizations of transfer functions in \cite{Ball_Staffans}.
Assume the operators $\tilde{A,}\tilde{B,}C,D$ formally satisfy the
equations in \cite[formula (1.7)]{Ball_Staffans}, i.e.,
\[
\tilde{A}+\tilde{A}^{*}=-\tilde{B}\tilde{B}^{*},\quad C=\tilde{B}^{*},\quad D=1.
\]
 Then by the skew-selfadjointness of $\tilde{A}$ we deduce that $0=\tilde{B}=C^{*}.$
With the notation from Theorem \ref{thm:conservative_control}, we
get that 
\[
\left(M_{1,22}^{-1}M_{1,20}\right)^{*}B_{2}=\left(1\cdot(-C_{0})\right)^{*}D=0=B_{0}
\]
and
\[
\left(M_{1,22}^{-1}M_{1,21}\right)^{*}B_{2}=\left(1\cdot(-C_{1})\right)^{*}D=0=B_{1},
\]
thus the operator equations of the above theorem are satisfied. The
corresponding control conservation equation reads
\[
\frac{1}{2}\left\langle \left.x\right|x\right\rangle (a)-\frac{1}{2}\left\langle \left.x\right|x\right\rangle (b)=\int_{[a,b]}\left(\left\langle \left.y\right|y\right\rangle -\left\langle \left.u\right|u\right\rangle \right)
\]
for a.e. $a,b\in\rga0$ with $a<b$. A more sophisticated example
will be discussed after the proof of Theorem \ref{thm:conservative_control}.\end{example}
\begin{proof}[Proof of Theorem \ref{thm:conservative_control}]
 Similarly to the proof of the conservation equation for evolutionary
systems, we show the conservation equation stated here for initial
data $(v_{0},w_{0},y_{0})\in U,$ where $U$ is chosen according to
the definition of regularizing systems. Hence, analogously to the
proof of Theorem \ref{thm:conservative_evo} we get that $(v,w,y)$
takes values in the domain of $\left(\begin{array}{ccc}
0 & -F^{*} & 0\\
F & 0 & 0\\
0 & 0 & 0
\end{array}\right)$ and that $M_{0}\left(\begin{array}{c}
v\\
w\\
y
\end{array}\right)$ is locally differentiable in $L_{\mathrm{loc}}^{2}\left(\rga{0},H\right).$
Let $\phi\in\interior C_{\infty}(\rga{0})$. Then, we obtain, similarly
to (\ref{eq:pre_cons_1}), the equation
\begin{align*}
 & \Re\left\langle \left.\phi\left(\begin{array}{c}
v\\
w\\
y
\end{array}\right)\right|\partial_{0}M_{0}\left(\left(\begin{array}{c}
v\\
w\\
y
\end{array}\right)-\chi_{\rga{0}}\otimes\left(\begin{array}{c}
v_{0}\\
w_{0}\\
y_{0}
\end{array}\right)\right)\right\rangle _{H_{\nu.0}(\mathbb{R},H)}\\
 & \quad+\left\langle \left.\phi\left(\begin{array}{c}
v\\
w\\
y
\end{array}\right)\right|\Re M_{1}\left(\begin{array}{c}
v\\
w\\
y
\end{array}\right)\right\rangle =\Re\left\langle \left.\phi\left(\begin{array}{c}
v\\
w
\end{array}\right)\right|\left(\begin{array}{c}
B_{0}\\
B_{1}
\end{array}\right)u\right\rangle _{H_{\nu.0}(\mathbb{R},H_{0}\oplus H_{1})}+\Re\left\langle \phi y|B_{2}u\right\rangle _{H_{\nu.0}(\mathbb{R},Y)}
\end{align*}
and hence
\begin{align}
\Re\left\langle \left.\left(\begin{array}{c}
v\\
w\\
y
\end{array}\right)\right|\partial_{0}M_{0}\left(\left(\begin{array}{c}
v\\
w\\
y
\end{array}\right)-\chi_{\rga{0}}\otimes\left(\begin{array}{c}
v_{0}\\
w_{0}\\
y_{0}
\end{array}\right)\right)\right\rangle _{H}+\left\langle \left.\left(\begin{array}{c}
v\\
w\\
y
\end{array}\right)\right|\Re M_{1}\left(\begin{array}{c}
v\\
w\\
y
\end{array}\right)\right\rangle _{H} & =\label{eq:eq_a_e}\\
\Re\left\langle \left.\left(\begin{array}{c}
v\\
w
\end{array}\right)\right|\left(\begin{array}{c}
B_{0}\\
B_{1}
\end{array}\right)u\right\rangle _{H_{0}\oplus H_{1}}+\Re\left\langle y|B_{2}u\right\rangle _{Y}\nonumber 
\end{align}
almost everywhere on $\rga0$. We aim to substitute $y$ in the mixed
term on the right-hand side. For this, consider the last row equation
of the general system
\[
M_{1,20}v+M_{1,21}w+M_{1,22}y=B_{2}u.
\]
Using that $M_{1,22}$ is continuously invertible due to the positive
definiteness constraint on $\nu M_{0}+\Re M_{1}$, we therefore get
that 
\begin{align*}
y & =-M_{1,22}^{-1}M_{1,20}v-M_{1,22}^{-1}M_{1,21}w+M_{1,22}^{-1}B_{2}u.
\end{align*}
Thus we have
\begin{align*}
\Re\langle B_{2}u|y\rangle_{Y} & =\Re\left\langle B_{2}u|-M_{1,22}^{-1}M_{1,20}v-M_{1,22}^{-1}M_{1,21}w+M_{1,22}^{-1}B_{2}u\right\rangle _{Y}\\
 & =\Re\left\langle B_{2}u|M_{1,22}^{-1}B_{2}u\right\rangle _{Y}-\Re\left\langle B_{2}u|M_{1,22}^{-1}M_{1,20}v+M_{1,22}^{-1}M_{1,21}w\right\rangle _{Y}.
\end{align*}
The first term on the right-hand side of (\ref{eq:eq_a_e}) may --
using the compatibility condition -- be computed as follows 
\begin{align*}
\Re\left\langle \left.\left(\begin{array}{c}
v\\
w
\end{array}\right)\right|\left(\begin{array}{c}
B_{0}\\
B_{1}
\end{array}\right)u\right\rangle _{H_{0}\oplus H_{1}} & =\Re\left\langle v|B_{0}u\right\rangle _{H_{0}}+\Re\left\langle w|B_{1}u\right\rangle _{H_{1}}\\
 & =\Re\left\langle M_{1,22}^{-1}M_{1,20}v|B_{2}u\right\rangle _{Y}+\Re\left\langle M_{1,22}^{-1}M_{1,21}w|B_{2}u\right\rangle _{Y}.
\end{align*}
Hence, 
\[
\Re\left\langle \left.\left(\begin{array}{c}
v\\
w
\end{array}\right)\right|\left(\begin{array}{c}
B_{0}\\
B_{1}
\end{array}\right)u\right\rangle _{H_{0}\oplus H_{1}}+\Re\left\langle y|B_{2}u\right\rangle _{Y}=\Re\left\langle B_{2}u|M_{1,22}^{-1}B_{2}u\right\rangle _{Y}.
\]
Now, integrating equation (\ref{eq:eq_a_e}) over $[a,b]$ yields
\begin{align*}
\frac{1}{2}\Re\left\langle \left.\left(\begin{array}{c}
v\\
w\\
y
\end{array}\right)\right|M_{0}\left(\begin{array}{c}
v\\
w\\
y
\end{array}\right)\right\rangle _{H}(a)-\frac{1}{2}\Re\left\langle \left.\left(\begin{array}{c}
v\\
w\\
y
\end{array}\right)\right|M_{0}\left(\begin{array}{c}
v\\
w\\
y
\end{array}\right)\right\rangle _{H}(b) & =\\
\int_{[a,b]}\left(\left\langle \left.\left(\begin{array}{c}
v\\
w\\
y
\end{array}\right)\right|\Re M_{1}\left(\begin{array}{c}
v\\
w\\
y
\end{array}\right)\right\rangle _{H}-\left\langle B_{2}u|\Re M_{1,22}^{-1}B_{2}u\right\rangle _{Y}\right)
\end{align*}
for all $a,b$ positive with $a<b$ . Using an approximation argument
as in the proof of Theorem \ref{thm:conservative_evo}, we get the
desired assertion.
\end{proof}

\begin{example}
In \cite{PiTroWau_2012} we studied the conservation property of the
following particular system, which is possible to deduce from the
(abstract) system treated in \cite{Weiss2003} (take $z\eqqcolon v$
and $\dot{z}\eqqcolon\zeta$): 
\begin{align*}
 & \left(\partial_{0}\left(\begin{array}{ccc}
1 & \left(\begin{array}{cc}
0 & 0\end{array}\right) & 0\\
\left(\begin{array}{c}
0\\
0
\end{array}\right) & \left(\begin{array}{cc}
1 & 0\\
0 & 0
\end{array}\right) & \left(\begin{array}{c}
0\\
0
\end{array}\right)\\
0 & \left(\begin{array}{cc}
0 & 0\end{array}\right) & 0
\end{array}\right)+\left(\begin{array}{ccc}
0 & \left(\begin{array}{cc}
0 & 0\end{array}\right) & 0\\
\left(\begin{array}{c}
0\\
0
\end{array}\right) & \left(\begin{array}{cc}
0 & 0\\
0 & 1
\end{array}\right) & \left(\begin{array}{c}
0\\
0
\end{array}\right)\\
0 & \left(\begin{array}{cc}
0 & \sqrt{2}\end{array}\right) & 1
\end{array}\right)\right.\\
 & +\left.\left(\begin{array}{ccc}
0 & \mathbb{DIV} & 0\\
\mathbb{GRAD} & \left(\begin{array}{cc}
0 & 0\\
0 & 0
\end{array}\right) & \left(\begin{array}{c}
0\\
0
\end{array}\right)\\
0 & \left(\begin{array}{cc}
0 & 0\end{array}\right) & 0
\end{array}\right)\right)\left(\begin{array}{c}
v\\
\left(\begin{array}{c}
\zeta\\
w
\end{array}\right)\\
y
\end{array}\right)\\
 & =\left(\begin{array}{c}
0\\
\left(\begin{array}{c}
0\\
-\sqrt{2}
\end{array}\right)\\
-1
\end{array}\right)u+\delta\otimes\left(\begin{array}{c}
z^{(1)}\\
\left(\begin{array}{c}
z^{(0)}\\
0
\end{array}\right)\\
0
\end{array}\right),
\end{align*}
where $\mathbb{GRAD}$ and $\mathbb{DIV}$ are suitable operators
such that $\mathbb{DIV}^{\ast}=-\mathbb{GRAD}.$ We remark here that
the notation $\mathbb{GRAD}$ and $\mathbb{DIV}$ serve as a reminder
of the fact that the former is the negative adjoint of the latter.
In \cite{PiTroWau_2012}, these operators are similarly constructed
as the operator $F$ and $-F^{*}$ in Section \ref{sub:Boundary-Control-and}.
We also refer to Section \ref{sub:Boundary-control-and-wave} equation
(\ref{eq:TW}) for a more specific example. It was shown that this
system is well-posed and locally regularizing. Furthermore the compatibility
conditions of Theorem \ref{thm:conservative_control} are satisfied
with 
\begin{align*}
M_{1,22} & =1,\quad M_{1,20}=0,\quad M_{1,21}=\left(\begin{array}{cc}
0 & \sqrt{2}\end{array}\right),\\
B_{0} & =0,\quad B_{1}=\left(\begin{array}{c}
0\\
-\sqrt{2}
\end{array}\right),\quad B_{2}=-1.
\end{align*}
 Thus, we end up with the conservation equation 
\[
\frac{1}{2}\left(|v(a)|^{2}+|\zeta(a)|^{2}\right)-\frac{1}{2}\left(|v(b)|^{2}+|\zeta(b)|^{2}\right)=\intop_{a}^{b}|w(t)|^{2}+\sqrt{2}\Re\langle w(t)|y(t)\rangle+|y(t)|^{2}-|u(t)|^{2}\mbox{ d}t.
\]
From the last row we read off the equation $\sqrt{2}w+y=-u$ and thus
$w=-\frac{1}{\sqrt{2}}(y+u).$ If we plug in this representation of
$w$ we get 
\[
\frac{1}{2}\left(|v(a)|^{2}+|\zeta(a)|^{2}\right)-\frac{1}{2}\left(|v(b)|^{2}+|\zeta(b)|^{2}\right)=\intop_{a}^{b}\frac{1}{2}|y(t)|^{2}-\frac{1}{2}|u(t)|^{2}\mbox{ d}t,
\]
which is the conservation equality in \cite[Corollary 1.5]{Weiss2003}. 
\end{example}

\section{Boundary Control\label{sec:Boundary-control}}

We shall now consider particular types of control equations involving
so-called boundary control. One may find the notion of boundary control
systems in the literature, see e.g.~\cite{Avrov,Malinen_Staffans_impedance,Malinen_Staffans_JDE}.
These are equations of the form
\[
u=Gx,\quad\dot{x}=Lx,\quad y=Kx
\]
subject to certain initial conditions for suitable linear operators
$G,L,K$ on suitable Hilbert spaces. The operators $G$ and $K$ are
thought of as trace mappings, where the first one is onto, and $L$
is assumed to be a generator of a $C_{0}$-semi-group if restricted
to the kernel of $G$. The precise (abstract) definition of the latter
operators is done with the help of so-called boundary triples. We
infer that these kind of boundary control systems are, if we focus
on well-posedness issues only, a mere non-homogeneous (abstract) Cauchy
problem. Indeed, using that $G$ is onto, we get $w$ such that $Gw=u$.
Introducing the new variable $\tilde{x}\coloneqq x-w\in N(G)$, we
arrive at the equation
\[
\dot{\tilde{x}}=L\tilde{x}-\dot{w}+Lw,
\]
which may be solved by the variation of constants formula. The output
$y$ can then be computed as follows $y=K(\tilde{x}+w)$. For a more
specific account of this strategy, we refer the reader to Section
\ref{sub:Maxwell's-Equation-with}. 

We will mainly focus on a class of boundary control systems where
both the equations on the boundary have terms of the input and output.
These are for example special types of port-Hamiltonian systems or
the control system discussed in \cite{Weiss2003}. Moreover, in the
later study, we will develop a framework that gives a possible generalization
of (a subclass of) port-Hamiltonian systems to more than one spatial
dimension.

As a first introductory example, we consider these types of port-Hamiltonian
systems (cf. e.g. \cite{JacZwart1,ZwartGor2010}).

\subsection{Port-Hamiltonian Systems\label{sub:Port-Hamiltonian-systems}}

The notion of port-Hamiltonian systems with boundary control and observation
as discussed in \cite[Section 11.2]{JacZwartIsem} can be described
as follows: Let $n\in\mathbb{N}$, $a,b\in\mathbb{R}$, $a<b$, $P_{0},P_{1}\in\mathbb{K}^{n\times n}$,
$\mathcal{H}\in L^{\infty}(\oi ab,\mathbb{K}^{n\times n})$, $W_{B},W_{C}\in\mathbb{K}^{n\times2n}$.
We assume the following: 

\begin{itemize}

\item $P_{1}$ is invertible and selfadjoint,

\item for a.e. $\zeta\in[a,b]$, we have $\mathcal{H}(\zeta)$ is
selfadjoint and there exist $m,M\in]0,\infty[$ such that for a.e.
$\zeta\in[a,b]$ we have $m\leq\mathcal{H}(\zeta)\leq M$,

\item $W_{B}$ and $W_{C}$ have full rank and $\left(\begin{array}{c}
W_{B}\\
W_{C}
\end{array}\right)$ is invertible.

\end{itemize} 

The authors of \cite{JacZwartIsem} considered the problem of finding
$(x,y)$ such that for given $x^{(0)}\in L^{2}(\oi ab,\mathbb{K}^{n})$
and $u\colon\rga0\to\mathbb{K}^{n}$ twice continuously differentiable
the following equations hold
\begin{align*}
\dot{x}(t)= & P_{1}\partial_{1}\mathcal{H}x(t)+P_{0}\mathcal{H}x(t)\\
u(t)= & W_{B}\frac{1}{\sqrt{2}}\left(\begin{array}{cc}
P_{1} & -P_{1}\\
1 & 1
\end{array}\right)\left(\begin{array}{c}
(\mathcal{H}x(t))(b)\\
(\mathcal{H}x(t))(a)
\end{array}\right)\\
y(t)= & W_{C}\frac{1}{\sqrt{2}}\left(\begin{array}{cc}
P_{1} & -P_{1}\\
1 & 1
\end{array}\right)\left(\begin{array}{c}
(\mathcal{H}x(t))(b)\\
(\mathcal{H}x(t))(a)
\end{array}\right)\\
x(0)= & x^{(0)},
\end{align*}
where $\partial_{1}$ is the distributional derivative with respect
to the spatial variable. Under particular assumptions on the matrices
involved a well-posedness result can be obtained by using $C_{0}$-semigroup
theory, see for instance \cite[Theorem 13.3.2]{JacZwartIsem}. Our
perspective to boundary control systems considers a particular subclass
of port-Hamiltonian (boundary control) systems. This subclass shows
the advantage that it can be generalized to an analogue of port-Hamiltonian
systems in more than one spatial dimension. The key assumption is
that $P_{1}$ is unitarily equivalent to a matrix of the form $\left(\begin{array}{cc}
0 & N^{*}\\
N & 0
\end{array}\right)$, where $N\in\mathbb{K}^{\ell\times\ell}$ with $2\ell=n$. Consequently,
$P_{1}\partial_{1}$ is replaced by $\left(\begin{array}{cc}
0 & \partial_{1}N^{*}\\
N\partial_{1} & 0
\end{array}\right)$ with suitable domain. The unknown $x$ decomposes into $(x_{0},x_{1})$.
Furthermore, we assume that we only control the boundary values of
$x_{1}$ and that the output is given in terms of the boundary values%
\footnote{This assumptions can be guaranteed for instance for the Timoshenko
beam equation, the vibrating string equation or the one-dimensional
heat equation with boundary control, \cite{JacZwartIsem}. It does,
however, not capture the one-dimensional transport equation.%
} of $x_{0}$. We are led to study the following problem, which corresponds
as we will see to port-Hamiltonian systems with boundary control and
observation as considered in \cite{JacZwartIsem} in a pure Hilbert
space setting provided our key assumptions are satisfied:

Let $\ell\in\mathbb{N}$, $N\in\mathbb{K}^{\ell\times\ell}$ invertible,
$n\coloneqq2\ell$, $M_{0}\in L(L^{2}(\oi ab,\mathbb{K}^{n}))$ selfadjoint
and strictly positive definite. Let $M_{1}\in L(L^{2}(\oi ab,\mathbb{K}^{n})\oplus\mathbb{K}^{4\ell})$
with the restriction of $\Re M_{1}$ to a linear mapping in $\mathbb{K}^{4\ell}$
assumed to be strictly positive definite, and $B_{0},B_{1}\in\mathbb{K}^{n\times n}$.
We define the operators
\begin{align*}
N\partial_{1}\colon H^{1}(\oi ab,\mathbb{K}^{\ell})\subseteq L^{2}(\oi ab,\mathbb{K}^{\ell}) & \to L^{2}(\oi ab,\mathbb{K}^{\ell})\\
f & \mapsto Nf',\\
\partial_{1}N^{*}\colon H^{1}(\oi ab,\mathbb{K}^{\ell})\subseteq L^{2}(\oi ab,\mathbb{K}^{\ell}) & \to H_{-1}(|\partial_{1}|+\i)\\
f & \mapsto(N^{*}f)'-(N^{*}f)(b)\cdot\delta_{b}+(N^{*}f)(a)\cdot\delta_{a}.
\end{align*}
The expression $N^{*}f(b)$ is well-defined by the 1-dimensional Sobolev
embedding theorem and 
\[
N^{*}f(b)\cdot\delta_{b}\colon H^{1}(\oi ab,\mathbb{K}^{\ell})\to\mathbb{K},\: g\mapsto\left\langle N^{*}f(b)|g(b)\right\rangle .
\]
We define the operator $C\colon H_{1}(|\partial_{1}|+\i)\to\mathbb{K}^{n},f\mapsto(-Nf(b),Nf(a))$,
in other words $C=\left(-N\delta_{b}\right)\oplus N\delta_{a}$. Identifying
$\mathbb{K}^{n}=\mathbb{K}^{\ell}\oplus\mathbb{K}^{\ell}$ with its
dual, we get $C^{\diamond}\colon\mathbb{K}^{\ell}\oplus\mathbb{K}^{\ell}\to H_{-1}(|\partial_{1}|+\i)$,
$(x,y)\mapsto-N^{*}x\cdot\delta_{b}+N^{*}y\cdot\delta_{a}$. 

We consider the following problem: Find $(x_{0},x_{1},w,y)\in H_{\nu,-1}(\mathbb{R};L^{2}(\oi ab,\mathbb{K}^{n})\oplus\mathbb{K}^{2n})$
such that for given $u\in H_{\nu,0}(\mathbb{R},\mathbb{K}^{n})$ and
$\xi_{0},\xi_{1}\in L^{2}(\oi ab,\mathbb{K}^{\ell})$ we have 
\begin{align}
 & \left(\partial_{0}\left(\begin{array}{ccc}
M_{0,00} & \left(\begin{array}{cc}
M_{0,01} & 0\end{array}\right) & 0\\
\left(\begin{array}{c}
M_{0,10}\\
0
\end{array}\right) & \left(\begin{array}{cc}
M_{0,11} & 0\\
0 & 0
\end{array}\right) & \left(\begin{array}{c}
0\\
0
\end{array}\right)\\
0 & \left(\begin{array}{cc}
0 & 0\end{array}\right) & 0
\end{array}\right)+M_{1}\right.\nonumber \\
 & \left.+\left(\begin{array}{ccc}
0 & \left(\begin{array}{cc}
-\partial_{1}N^{*} & C^{\diamond}\end{array}\right) & 0\\
\left(\begin{array}{c}
-N\partial_{1}\\
-C
\end{array}\right) & \left(\begin{array}{cc}
0 & 0\\
0 & 0
\end{array}\right) & \left(\begin{array}{c}
0\\
0
\end{array}\right)\\
0 & \left(\begin{array}{cc}
0 & 0\end{array}\right) & 0
\end{array}\right)\right)\left(\begin{array}{c}
x_{0}\\
\left(\begin{array}{c}
x_{1}\\
w
\end{array}\right)\\
y
\end{array}\right)=\delta\otimes\left(\begin{array}{c}
\xi_{0}\\
\xi_{1}\\
0\\
0
\end{array}\right)+\left(\begin{array}{c}
0\\
0\\
B_{1}u\\
B_{2}u
\end{array}\right).\label{eq:port_Ham}
\end{align}
In Section \ref{sub:Boundary-Control-and} we shall see that this
type of problem is well-posed in $H_{\nu,-1}(\mathbb{R},L^{2}(\oi ab,\mathbb{K}^{n})\oplus\mathbb{K}^{2n})$.
For convenience%
\footnote{This holds true if we assume the initial data $\xi_{0}$, $\xi_{1}$
and the control $u$ to be smooth enough.%
}, assume that $(x_{0},x_{1},w,y)\in H_{\nu,0}(\mathbb{R},L^{2}(\oi ab,\mathbb{K}^{n})\oplus\mathbb{K}^{2n})$
is a solution of the above system. Then, it follows that $\left(\begin{array}{ccc}
0 & \left(\begin{array}{cc}
-\partial_{1}N^{*} & C^{\diamond}\end{array}\right) & 0\\
\left(\begin{array}{c}
-N\partial_{1}\\
-C
\end{array}\right) & \left(\begin{array}{cc}
0 & 0\\
0 & 0
\end{array}\right) & \left(\begin{array}{c}
0\\
0
\end{array}\right)\\
0 & \left(\begin{array}{cc}
0 & 0\end{array}\right) & 0
\end{array}\right)\left(\begin{array}{c}
x_{0}\\
\left(\begin{array}{c}
x_{1}\\
w
\end{array}\right)\\
y
\end{array}\right)$ is an element of $H_{\nu,-1}(\mathbb{R},L^{2}(\oi ab,\mathbb{K}^{n})\oplus\mathbb{K}^{2n})$.
Consequently, we get that 
\[
\left(\begin{array}{cc}
-\partial_{1}N^{*} & C^{\diamond}\end{array}\right)\left(\begin{array}{c}
x_{1}\\
w
\end{array}\right)\in H_{\nu,-1}(\mathbb{R},L^{2}(\oi ab,\mathbb{K}^{\ell})).
\]
Thus, with $w=(w_{1},w_{2})$ 
\[
-N^{*}x'_{1}+(N^{*}x_{1})(b)\cdot\delta_{b}-(N^{*}x_{1})(a)\cdot\delta_{a}-N^{*}w_{1}\delta_{b}+N^{*}w_{2}\delta_{a}\in H_{\nu,-1}(\mathbb{R},L^{2}(\oi ab,\mathbb{K}^{\ell})).
\]
The latter, however, can only happen if $x_{1}(b)=w_{1}$ and $x_{1}(a)=w_{2}$.
Hence, the first two equations read as
\begin{align*}
 & \left(\partial_{0}\left(\begin{array}{cccc}
M_{0,00} & M_{0,01} & 0 & 0\\
M_{0,10} & M_{0,11} & 0 & 0
\end{array}\right)+\left(\begin{array}{cccc}
M_{1,00} & M_{1,01} & M_{1,02} & M_{1,03}\\
M_{1,10} & M_{1,11} & M_{1,12} & M_{1,13}
\end{array}\right)\right.\\
 & \left.\quad+\left(\begin{array}{cccc}
0 & -\partial_{1}N^{*} & C^{\diamond} & 0\\
-N\partial_{1} & 0 & 0 & 0
\end{array}\right)\right)\left(\begin{array}{c}
x_{0}\\
\begin{array}{c}
x_{1}\\
w
\end{array}\\
y
\end{array}\right)=\delta\otimes\left(\begin{array}{c}
\xi_{0}\\
\xi_{1}
\end{array}\right),
\end{align*}
or
\[
\partial_{0}\left(\begin{array}{cc}
M_{0,00} & M_{0,01}\\
M_{0,10} & M_{0,11}
\end{array}\right)\left(\begin{array}{c}
x_{0}\\
x_{1}
\end{array}\right)+\left(\begin{array}{cccc}
M_{1,00} & M_{1,01} & M_{1,02} & M_{1,03}\\
M_{1,10} & M_{1,11} & M_{1,12} & M_{1,13}
\end{array}\right)\left(\begin{array}{c}
x_{0}\\
\begin{array}{c}
x_{1}\\
w
\end{array}\\
y
\end{array}\right)+\left(\begin{array}{c}
-N^{*}x_{1}'\\
-Nx_{0}'
\end{array}\right)=\delta\otimes\left(\begin{array}{c}
\xi_{0}\\
\xi_{1}
\end{array}\right).
\]
Thus, we arrive at the following system
\begin{align*}
 & \partial_{0}\left(\begin{array}{cc}
M_{0,00} & M_{0,01}\\
M_{0,10} & M_{0,11}
\end{array}\right)\left(\begin{array}{c}
x_{0}\\
x_{1}
\end{array}\right)\\
 & \quad+\left(\begin{array}{cccc}
M_{1,00} & M_{1,01} & M_{1,02} & M_{1,03}\\
M_{1,10} & M_{1,11} & M_{1,12} & M_{1,13}
\end{array}\right)\left(\begin{array}{c}
x_{0}\\
\begin{array}{c}
x_{1}\\
w
\end{array}\\
y
\end{array}\right)-\left(\begin{array}{cc}
0 & N^{*}\\
N & 0
\end{array}\right)\partial_{1}\left(\begin{array}{c}
x_{0}\\
x_{1}
\end{array}\right)=\delta\otimes\left(\begin{array}{c}
\xi_{0}\\
\xi_{1}
\end{array}\right)
\end{align*}
In order to reproduce the formal structure of port-Hamiltonian systems,
we are led to assume that $\left(\begin{array}{cc}
M_{0,00} & M_{0,01}\\
M_{0,10} & M_{0,11}
\end{array}\right)=\mathcal{H}^{-1}$ and $\left(\begin{array}{cc}
M_{1,00} & M_{1,01}\\
M_{1,10} & M_{1,11}
\end{array}\right)=-P_{0}.$ Moreover, $\left(\begin{array}{cc}
M_{1,02} & M_{1,03}\\
M_{1,12} & M_{1,13}
\end{array}\right)$ must be assumed to be $0.$ To simplify matters further, we consider
the second two rows of $M_{1}$ to be of the form 
\[
\left(\begin{array}{cccc}
0 & 0 & M_{1,22} & M_{1,23}\\
0 & 0 & M_{1,32} & M_{1,33}
\end{array}\right).
\]
Then the second two rows of system (\ref{eq:port_Ham}) are
\begin{align*}
M_{1,22}w+M_{1,23}y-Cx_{0} & =B_{1}u\\
M_{1,32}w+M_{1,33}y & =B_{2}u.
\end{align*}
Using the above condition that $\left(\begin{array}{c}
w_{1}\\
w_{2}
\end{array}\right)=\left(\begin{array}{c}
x_{1}(b)\\
x_{1}(a)
\end{array}\right),$ we get that 
\begin{align*}
M_{1,22}\left(\begin{array}{c}
x_{1}(b)\\
x_{1}(a)
\end{array}\right)+M_{1,23}y+\left(\begin{array}{c}
Nx_{0}(b)\\
-Nx_{0}(a)
\end{array}\right) & =B_{1}u\\
M_{1,32}\left(\begin{array}{c}
x_{1}(b)\\
x_{1}(a)
\end{array}\right)+M_{1,33}y & =B_{2}u.
\end{align*}
In the spirit of boundary control and boundary observation we have
that the boundary values of $x_{0}$ are expressed as a linear combination
of the output $y$. Thus, there is a linear operator $W\in L(H_{\nu,0}(\mathbb{R},\mathbb{K}^{n}))$
such that $Wy=\chi_{]0,\infty[}(m_{0})\left(\begin{array}{c}
Nx_{0}(b)\\
-Nx_{0}(a)
\end{array}\right)$. Moreover, assuming suitable invertibility properties on the operators
$B_{1},B_{2},M_{1,33}$ and $M_{1,23}$, we may express the above
two equations as a system of two equations of the form: 
\begin{align*}
u & =(B_{1}-(M_{1,23}+W)M_{1,33}^{-1}B_{2})^{-1}\left(M_{1,22}-(M_{1,23}+W)M_{1,33}^{-1}M_{1,32}\right)\left(\begin{array}{c}
x_{1}(b)\\
x_{1}(a)
\end{array}\right),\\
y & =\left(B_{1}B_{2}^{-1}M_{1,33}-(M_{1,23}+W)\right)^{-1}\left(M_{1,22}-B_{1}B_{2}^{-1}M_{1,32}\right)\left(\begin{array}{c}
x_{1}(b)\\
x_{1}(a)
\end{array}\right).
\end{align*}
These equations are the control and the observation equations and
they are of the same form as considered in \cite{JacZwartIsem}. A
similar reasoning is applied in Remark \ref{control_observation_equation},
where a more general situation is considered. 

The discussion of boundary control within the context of port-Hamiltonian
systems becomes accessible due to the Sobolev-embedding theorem yielding
a continuous boundary trace operator and a finite-dimensional boundary
trace space. In higher-dimensional situations the Sobolev-embedding
theorem depends on the geometry of the underlying domain. A continuous
boundary trace operator can only be defined for domains satisfying
some regularity assumptions at the boundary, e.g.\ assuming a Lipschitz-continuous
boundary. We shall approach boundary control systems from a more general
perspective without assuming undue regularity of the boundary. In
order to have the functional analytic notions at hand to replace the
boundary trace space by an appropriate alias that captures the boundary
data, we implement the necessary concepts in the next section.

~

\subsection{Boundary Data Spaces\label{sub:Boundary-Data-Spaces}}

Throughout this section, let $H_{0}$ and $H_{1}$ be Hilbert spaces
and let%
\footnote{The notation $\interior G,\:\interior D$ is chosen as a reminder
of the basic situation taking these as the closure of the classical
operations $\grad$ and $\dive$ defined on $C_{\infty}-$functions
with compact support in an open set $\Omega$ of $\mathbb{R}^{n},$
$n\in\mathbb{N}.$ In other practical cases, these operators can change
role or can be totally different operators such as $\curl.$%
} $\interior G\subseteq H_{0}\oplus H_{1},\:\interior D\subseteq H_{1}\oplus H_{0}$
be two densely defined, closed linear operators, which are assumed
to be formally skew-adjoint linear operators, i.e.
\begin{align*}
\interior D & \subseteq D\coloneqq-\left(\interior G\right)^{*},\\
\interior G & \subseteq G\coloneqq-\left(\interior D\right)^{*}.
\end{align*}

\begin{lem}
We have the orthogonal decompositions
\begin{align}
H_{1}\left(\left|G\right|+\i\right) & =H_{1}\left(\left|\interior G\right|+\i\right)\oplus N\left(1-DG\right),\label{eq:deco1}\\
H_{1}\left(\left|D\right|+\i\right) & =H_{1}\left(\left|\interior D\right|+\i\right)\oplus N\left(1-GD\right).\label{eq:deco2}
\end{align}
\end{lem}
\begin{proof}
Let $\phi\in H_{1}\left(\left|\interior G\right|+\i\right)^{\perp}$.
Then for all $\psi\in H_{1}\left(\left|\interior G\right|+\i\right)$
\begin{align*}
0 & =\left\langle \psi|\phi\right\rangle _{H_{1}\left(\left|G\right|+\i\right)}\\
 & =\left\langle \psi|\phi\right\rangle _{H_{0}}+\left\langle \left|G\right|\psi|\left|G\right|\phi\right\rangle _{H_{0}}\\
 & =\left\langle \psi|\phi\right\rangle _{H_{0}}+\left\langle G\psi|G\phi\right\rangle _{H_{1}}\\
 & =\left\langle \psi|\phi\right\rangle _{H_{0}}+\left\langle \interior G\psi|G\phi\right\rangle _{H_{1}}
\end{align*}
We read off that $G\phi\in D\left(\left(\interior G\right)^{*}\right)=D\left(D\right)$
and 
\[
DG\phi=\phi.
\]
The remaining case follows analogously.
\end{proof}
We define%
\footnote{The notation $BD\left(\:\cdot\:\right)$ is supposed to be a reminder
that in applications these spaces will serve as the spaces of boundary
data.%
}
\begin{align*}
BD\left(G\right) & \coloneqq N\left(1-DG\right)\\
BD\left(D\right) & \coloneqq N\left(1-GD\right)
\end{align*}
and obtain 
\begin{align*}
G\left[BD\left(G\right)\right] & \subseteq BD\left(D\right),\\
D\left[BD\left(D\right)\right] & \subseteq BD\left(G\right).
\end{align*}

For later purposes we also introduce the canonical projectors $\pi_{BD\left(G\right)}:H_{1}\left(\left|G\right|+\i\right)\to BD\left(G\right)$
and $\pi_{BD\left(D\right)}:H_{1}\left(\left|D\right|+\i\right)\to BD\left(D\right)$
onto the component spaces $BD\left(G\right),\: BD\left(D\right)$
according to the direct sum decompositions (\ref{eq:deco1}), (\ref{eq:deco2}),
respectively. The orthogonal projectors $P_{BD\left(G\right)}:H_{1}\left(\left|G\right|+\i\right)\to H_{1}\left(\left|G\right|+\i\right)$,
$P_{BD\left(D\right)}:H_{1}\left(\left|G\right|+\i\right)\to H_{1}\left(\left|G\right|+\i\right)$
associated with (\ref{eq:deco1}) and (\ref{eq:deco2}) can now be
expressed as
\[
P_{BD\left(G\right)}=\pi_{BD\left(G\right)}^{*}\pi_{BD\left(G\right)},\; P_{BD\left(D\right)}=\pi_{BD\left(D\right)}^{*}\pi_{BD\left(D\right)}.
\]
Note that $\pi_{BD\left(G\right)}^{*},\:\pi_{BD\left(D\right)}^{*}$
are the canonical embeddings of $BD\left(G\right)$ in $H_{1}\left(\left|G\right|+\i\right)$
and of $BD\left(D\right)$ in $H_{1}\left(\left|D\right|+\i\right),$
respectively.

Thus, on $BD\left(D\right)$ we may define the operator $\overset{\bullet}{D}$
by%
\footnote{These operators are an abstract version of the Dirichlet-to-Neumann
operator since the ``boundary data'' space for $G$ is transformed
into the ``boundary data'' space for $D$. Indeed, if $u$ is a
solution of the inhomogeneous ``Dirichlet boundary value problem''
\begin{align*}
\left(1-DG\right)u & =0\\
u-g & \in D\left(\interior G\right)
\end{align*}
for given data $g\in BD\left(G\right)$ then also
\[
\left(1-D\interior G\right)\left(u-g\right)=0
\]
implying
\[
u=g.
\]
This implies
\[
\overset{\bullet}{G}u=\overset{\bullet}{G}g
\]
and $u$ is therefore also the solution of the inhomogeneous ``Neumann
boundary value problem''
\begin{align*}
\left(1-DG\right)u & =0\\
Gu-\overset{\bullet}{G}g & \in D\left(\interior D\right)
\end{align*}
and vice versa.%
}
\begin{align*}
\overset{\bullet}{D}:BD\left(D\right) & \to BD\left(G\right)\\
\phi & \mapsto D\phi
\end{align*}
and the operator $\overset{\bullet}{G}$ by
\begin{align*}
\overset{\bullet}{G}:BD\left(G\right) & \to BD\left(D\right)\\
\phi & \mapsto G\phi.
\end{align*}
The operators $\overset{\bullet}{D}$ and $\overset{\bullet}{G}$
enjoy the following surprising property.
\begin{thm}
\label{thm:unitary_D_G}We have that%
\footnote{Note, however, that in contrast we have
\[
\left(G\right)^{*}=-\interior D
\]
 in $H_{0}\left(\left|\interior D\right|+\i\right)\oplus H_{0}\left(\left|G\right|+\i\right).$%
}
\[
\left(\overset{\bullet}{G}\right)^{*}=\overset{\bullet}{D}=\left(\overset{\bullet}{G}\right)^{-1}.
\]
In particular, $\overset{\bullet}{G}$ and $\overset{\bullet}{D}$
are unitary.\end{thm}
\begin{proof}
Obviously is $\overset{\bullet}{D}\overset{\bullet}{G}$ the identity
on $BD\left(G\right)$ and $\overset{\bullet}{G}\overset{\bullet}{D}$
the identity on $BD\left(D\right).$ Consequently, 
\[
\overset{\bullet}{D}=\left(\overset{\bullet}{G}\right)^{-1}.
\]
Moreover, for $\phi\in BD\left(G\right)$ and $\psi\in BD\left(D\right)$
\begin{align*}
\left\langle \left.\overset{\bullet}{G}\phi\right|\psi\right\rangle _{BD(D)} & \coloneqq\left\langle \left.\overset{\bullet}{G}\phi\right|\psi\right\rangle _{H_{1}(|D|+\i)}=\left\langle \left.\overset{\bullet}{G}\phi\right|\psi\right\rangle _{H_{0}(|D|+\i)}+\left\langle \left.\overset{\bullet}{D}\overset{\bullet}{G}\phi\right|\overset{\bullet}{D}\psi\right\rangle _{H_{0}(|G|+\i)}\\
 & =\left\langle \left.\overset{\bullet}{G}\phi\right|\overset{\bullet}{G}\overset{\bullet}{D}\psi\right\rangle _{H_{0}(|D|+\i)}+\left\langle \phi\left|\overset{\bullet}{D}\psi\right.\right\rangle _{H_{0}(|G|+\i)}\\
 & =\left\langle \phi\left|\overset{\bullet}{D}\psi\right.\right\rangle _{H_{1}(|G|+\i)}\eqqcolon\left\langle \phi\left|\overset{\bullet}{D}\psi\right.\right\rangle _{BD(G)}
\end{align*}
leading to 
\[
\left(\overset{\bullet}{G}\right)^{*}=\overset{\bullet}{D}
\]
in $BD\left(D\right)\oplus BD\left(G\right).$\end{proof}
\begin{example}
\label{As-an-application} As an application let us calculate the
dual mapping $\pi_{BD\left(G\right)}^{\diamond}$ of 
\[
\pi_{BD\left(G\right)}:H_{1}\left(\left|G\right|+\i\right)\to BD\left(G\right)
\]
according to the Gelfand triplet $H_{1}(|G|+\i)\subseteq H_{0}(|G|+\i)\subseteq H_{-1}(|G|+\i)$%
\footnote{Note that the Riesz-mapping $R_{H_{1}(|G|+\i)}:H_{-1}(|G|+\i)\to H_{1}(|G|+\i)$
is given by $R_{H_{1}(|G|+\i)}\phi=(1+|G|^{2})^{-1}\phi=(1+G^{*}G)^{-1}\phi=(1-\interior DG)^{-1}\phi$.%
}, which would be a mapping from $BD\left(G\right)$ (identified with
$BD\left(G\right)^{*}$) into $H_{-1}\left(\left|G\right|+\i\right)$.
We find
\begin{align*}
\left(\pi_{BD\left(G\right)}\right)^{\diamond} & =R_{H_{1}\left(\left|G\right|+\i\right)}^{*}\pi_{BD\left(G\right)}^{*}\\
 & =\left(\left|G\right|^{2}+1\right)\pi_{BD\left(G\right)}^{*}\\
 & =\pi_{BD\left(G\right)}^{*}-\interior DG\pi_{BD\left(G\right)}^{*}\\
 & =\pi_{BD\left(G\right)}^{*}-\interior D\pi_{BD\left(D\right)}^{*}\overset{\bullet}{G}.
\end{align*}

\end{example}
\begin{rem}
In the literature, in order to discuss boundary control systems in
an operator-theoretic framework, the concept of boundary triples is
used, see e.g.~\cite{Malinen_Staffans_JDE,Behrndt_bdy_rel,Behrndt_Kreusler_bdy_rel,Derkach_bdy_relations},
we also refer to \cite{BDY_systems,SWIPs}, where in \cite{BDY_systems}
a unified perspective is given. A boundary triple is a symmetric operator
$S$ defined in a Hilbert space $H$ and two continuous linear operators
$\Gamma_{0},\Gamma_{1}\colon H_{1}(\left|S^{*}\right|+\i)\to K,$
mapping onto a Hilbert space $K$. Moreover, for all $x,y\in D(S^{*})$
the following equality should be satisfied
\[
\langle S^{*}x|y\rangle_{H}-\langle x|S^{*}y\rangle_{H}=\langle\Gamma_{0}x|\Gamma_{1}y\rangle_{K}-\langle\Gamma_{1}x|\Gamma_{0}y\rangle_{K}.
\]
In the literature one finds the notation $(K,\Gamma_{0},\Gamma_{1})$,
which explains the name. In the situation of this section we also
have a boundary triple: Setting 
\[
S=-\i\left(\begin{array}{cc}
0 & \interior D\\
\interior G & 0
\end{array}\right),\quad K=BD(G),\quad\Gamma_{0}=\left(\begin{array}{cc}
\pi_{BD(G)} & 0\end{array}\right),\quad\Gamma_{1}=\left(\begin{array}{cc}
0 & \i\overset{\bullet}{D}\pi_{BD(D)}\end{array}\right),
\]
we get a boundary triple. Indeed, let $(u,v),(x,y)\in H_{1}(\left|S^{*}\right|+\i)=H_{1}(\left|G\right|+\i)\oplus H_{1}(\left|D\right|+\i)$.
Denoting $P_{\interior D}\coloneqq1-P_{BD(D)}$ and $P_{\interior G}\coloneqq1-P_{BD(G)}$,
we compute
\begin{align*}
 & -\left\langle S^{*}\left(\begin{array}{c}
u\\
v
\end{array}\right)\left|\left(\begin{array}{c}
x\\
y
\end{array}\right)\right.\right\rangle _{H_{0}(|S^{\ast}|+\i)}+\left\langle \left(\begin{array}{c}
u\\
v
\end{array}\right)\left|S^{*}\left(\begin{array}{c}
x\\
y
\end{array}\right)\right.\right\rangle _{H_{0}(|S^{\ast}|+\i)}\\
 & =\i\left(\langle Dv|x\rangle_{H_{0}(|G|+\i)}+\langle Gu|y\rangle_{H_{0}(|D|+\i)}+\langle u|Dy\rangle_{H_{0}(|G|+\i)}+\langle v|Gx\rangle_{H_{0}(|D|+\i)}\right)\\
 & =\i\left(\langle DP_{\interior D}v+DP_{BD(D)}v|x\rangle_{H_{0}(|G|+\i)}+\langle GP_{\interior G}u+GP_{BD(G)}u|y\rangle_{H_{0}(|D|+\i)}\right.\\
 & \quad\left.+\langle P_{\interior G}u+P_{BD(G)}u|Dy\rangle_{H_{0}(|G|+\i)}+\langle P_{\interior D}v+P_{BD(D)}v|Gx\rangle_{H_{0}(|D|+\i)}\right)\\
 & =\i\left(\langle DP_{BD(D)}v|x\rangle_{H_{0}(|G|+\i)}+\langle GP_{BD(G)}u|y\rangle_{H_{0}(|D|+\i)}\right.\\
 & \quad\left.+\langle P_{BD(G)}u|Dy\rangle_{H_{0}(|G|+\i)}+\langle P_{BD(D)}v|Gx\rangle_{H_{0}(|D|+\i)}\right)\\
 & =\i\left(\langle DP_{BD(D)}v|x\rangle_{H_{0}(|G|+\i)}+\langle GP_{BD(G)}u|y\rangle_{H_{0}(|D|+\i)}\right.\\
 & \quad\left.+\langle DGP_{BD(G)}u|Dy\rangle_{H_{0}(|G|+\i)}+\langle GDP_{BD(D)}v|Gx\rangle_{H_{0}(|D|+\i)}\right)\\
 & =\i\left(\langle DP_{BD(D)}v|x\rangle_{H_{1}(|G|+\i)}+\langle GP_{BD(G)}u|y\rangle_{H_{1}(|D|+\i)}\right)\\
 & =\i\left(\langle\overset{\bullet}{D}\pi_{BD(D)}v|\pi_{BD(G)}x\rangle_{BD(G)}+\langle\overset{\bullet}{G}\pi_{BD(G)}u|\pi_{BD(D)}y\rangle_{BD(D)}\right)\\
 & =\i\left(\langle\overset{\bullet}{D}\pi_{BD(D)}v|\pi_{BD(G)}x\rangle_{BD(G)}+\langle\pi_{BD(G)}u|\overset{\bullet}{D}\pi_{BD(D)}y\rangle_{BD(G)}\right)\\
 & =-\langle\i\overset{\bullet}{D}\pi_{BD(D)}v|\pi_{BD(G)}x\rangle_{BD(G)}+\langle\pi_{BD(G)}u|\i\overset{\bullet}{D}\pi_{BD(D)}y\rangle_{BD(G)}.
\end{align*}
 
\end{rem}

\subsection{Control Systems with Boundary Control and Boundary Observation\label{sub:Boundary-Control-and}}

We apply our previous findings in this section to model problems with
boundary control and boundary observation in more complex situations.
For this purpose we consider abstract linear control systems $\mathcal{C}_{M_{0},M_{1},F,B}$
where the operator $F$ is given in the following form 
\begin{equation}
F\coloneqq\left(\begin{array}{c}
-G\\
C
\end{array}\right):H_{1}(|G|+\i)\subseteq H_{0}(|G|+\i)\to H_{0}(|\interior D|+\i)\oplus V,\label{eq:op_F}
\end{equation}

with $C\in L(H_{1}(|G|+\i),V)$ for some Hilbert space $V$ and $G,D$
are as in Subsection \ref{sub:Boundary-Data-Spaces}. As a variant
of \cite[Lemma 5.1]{PiTroWau_2012} we compute the adjoint of $F$
explicitly under the additional constraint that $G$ is boundedly
invertible.
\begin{thm}
\label{thm:adjoint_F}Let $F$ be given as above and let $G$ be boundedly
invertible. Then 
\begin{align*}
F^{\ast}:D(F^{\ast})\subseteq H_{0}(|\interior D|+\i)\oplus V & \to H_{0}(|G|+\i)\\
(\zeta,w) & \mapsto\interior D\zeta+C^{\diamond}w,
\end{align*}
where $C^{\diamond}$ is the dual operator of $C$ with respect to
the Gelfand-triplet $H_{1}(|G|+\i)\subseteq H_{0}(|G|+\i)\subseteq H_{-1}(|G|+\i)$
and 
\[
D(F^{\ast})=\{(\zeta,w)\in H_{0}(|\interior D|+\i)\oplus V\,|\,\interior D\zeta+C^{\diamond}w\in H_{0}(|G|+\i)\}.
\]
\end{thm}
\begin{proof}
We define
\begin{align*}
K:D(K)\subseteq H_{0}(|\interior D|+\i)\oplus V & \to H_{0}(|G|+\i)\\
(\zeta,w) & \mapsto\interior D\zeta+C^{\diamond}w,
\end{align*}
 with $D(K)\coloneqq\{(\zeta,w)\in H_{0}(|\interior D|+\i)\oplus V\,|\,\interior D\zeta+C^{\diamond}w\in H_{0}(|G|+\i)\}.$
From 
\[
\left(\left(\begin{array}{cc}
\interior D & C^{\ast}\end{array}\right):H_{1}(|\interior D|+\i)\oplus V\subseteq H_{0}(|\interior D|+\i)\oplus V\to H_{0}(|G|+\i)\right)\subseteq K,
\]
we get that $K$ is densely defined. Furthermore $K$ is closed. Thus,
it suffices to prove $K^{\ast}=F$. Let $v\in D(K^{\ast}).$ Then
there exists $\left(\begin{array}{c}
f\\
g
\end{array}\right)\in H_{0}(|\interior D|+\i)\oplus V$ such that for all $\left(\begin{array}{c}
\zeta\\
w
\end{array}\right)\in D(K)$ we have 
\[
\left\langle \left.K\left(\begin{array}{c}
\zeta\\
w
\end{array}\right)\right|v\right\rangle _{H_{0}(|G|+\i)}=\left\langle \left.\left(\begin{array}{c}
\zeta\\
w
\end{array}\right)\right|\left(\begin{array}{c}
f\\
g
\end{array}\right)\right\rangle _{H_{0}(|\interior D|+\i)\oplus V}.
\]
Choosing $w=0$ and $\zeta\in H_{1}(|\interior D|+\i)$ we get 
\[
\langle\interior D\zeta|v\rangle_{H_{0}(|G|+\i)}=\langle\zeta|f\rangle_{H_{0}(|\interior D|+\i)},
\]
yielding $v\in H_{1}(|G|+\i)$ and $f=-Gv.$ Let now $w\in V$ be
arbitrarily chosen. Like in \cite[Theorem 2.1.4]{TroWau_2012} we
find an element $\zeta\in H_{0}(|\interior D|+\i)$ such that $\interior D\zeta=-C^{\diamond}w.$
For this choice of $\zeta$ we get $(\zeta,w)\in D(K)$ with $K\left(\begin{array}{c}
\zeta\\
w
\end{array}\right)=0$ and thus we compute
\begin{align*}
0 & =\left\langle \left.\left(\begin{array}{c}
\zeta\\
w
\end{array}\right)\right|\left(\begin{array}{c}
-Gv\\
g
\end{array}\right)\right\rangle _{H_{0}(|\interior D|+\i)\oplus V}\\
 & =\langle\zeta|-Gv\rangle_{H_{0}(|\interior D|+\i)}+\langle w|g\rangle_{V}\\
 & =\langle\interior D\zeta|v\rangle_{H_{0}(|G|+\i)}+\langle w|g\rangle_{V}\\
 & =\langle-C^{\diamond}w|v\rangle_{H_{0}(|G|+\i)}+\langle w|g\rangle_{V}\\
 & =\langle w|-Cv+g\rangle_{V}.
\end{align*}
This shows $g=Cv$ and hence $K^{\ast}\subseteq F.$ Let now $v\in D(F)$
and $\left(\begin{array}{c}
\zeta\\
w
\end{array}\right)\in D(K).$ Then 
\begin{align*}
\left\langle \left.K\left(\begin{array}{c}
\zeta\\
w
\end{array}\right)\right|v\right\rangle _{H_{0}(|G|+\i)} & =\langle\interior D\zeta+C^{\diamond}w|v\rangle_{H_{0}(|G|+\i)}\\
 & =\langle\interior D\zeta|v\rangle_{H_{0}(|G|+\i)}+\langle C^{\diamond}w|v\rangle_{H_{0}(|G|+\i)}\\
 & =\langle\zeta|-Gv\rangle_{H_{0}(|\interior D|+\i)}+\langle w|Cv\rangle_{V},
\end{align*}
which shows $F\subseteq K^{\ast}.$ \end{proof}
\begin{rem}
\label{control_observation_equation}With this choice of $F$ we can
model systems with boundary observation and boundary control in the
following way: Let $M_{0}$ and $M_{1}$ be of the following form
\[
M_{0}=\left(\begin{array}{cccc}
M_{0,00} & M_{0,01} & 0 & 0\\
M_{0,10} & M_{0,11} & 0 & 0\\
0 & 0 & 0 & 0\\
0 & 0 & 0 & 0
\end{array}\right),\: M_{1}=\left(\begin{array}{cccc}
M_{1,00} & M_{1,01} & M_{1,02} & M_{1,03}\\
M_{1,10} & M_{1,11} & M_{1,12} & M_{1,13}\\
M_{1,20} & M_{1,21} & M_{1,22} & M_{1,23}\\
M_{1,30} & M_{1,31} & M_{1,32} & M_{1,33}
\end{array}\right)
\]
for suitable bounded linear operator $M_{i,jk}$ such that $M_{0}$
is selfadjoint and $\nu M_{0}+\Re M_{1}$ is uniformly strictly positive
definite for all $\nu\in]0,\infty[$ sufficiently large. Consider
the abstract linear control system
\begin{equation}
\left(\partial_{0}M_{0}+M_{1}+\left(\begin{array}{ccc}
0 & -F^{\ast} & 0\\
F & 0 & 0\\
0 & 0 & 0
\end{array}\right)\right)\left(\begin{array}{c}
v\\
\left(\begin{array}{c}
\zeta\\
w
\end{array}\right)\\
y
\end{array}\right)=\delta\otimes M_{0}\left(\begin{array}{c}
v_{0}\\
\left(\begin{array}{c}
\zeta_{0}\\
w_{0}
\end{array}\right)\\
y_{0}
\end{array}\right)+\left(\begin{array}{c}
0\\
\left(\begin{array}{c}
0\\
B_{1}
\end{array}\right)\\
B_{2}
\end{array}\right)u,\label{eq:model_bdy_ob_con}
\end{equation}
where $F$ is chosen as in (\ref{eq:op_F}) and $B_{1}\in L(U,V),\, B_{2}\in L(U,Y).$
We characterize the domain of $F^{\ast}.$ By Theorem \ref{thm:adjoint_F}
a pair $\left(\zeta,w\right)$ belongs to $D(F^{\ast})$ if and only
if $\interior D\zeta+C^{\diamond}w\in H_{0}(|G|+\i).$ Using the invertibility
of $\interior D$ on the related Sobolev chains this is equivalent
to 
\[
\zeta+\interior D^{-1}C^{\diamond}w\in H_{1}(|\interior D|+\i).
\]
Hence, using the results on boundary data spaces this reads as 
\begin{equation}
\pi_{BD(D)}(\zeta+\interior D^{-1}C^{\diamond}w)=0.\label{eq:bd}
\end{equation}
This means that $w$ prescribes the boundary data of $\zeta.$ We
read off the last two lines of equation (\ref{eq:model_bdy_ob_con})
and get 
\begin{align*}
M_{1,20}v+M_{1,21}\zeta+M_{1,22}w+M_{1,23}y+Cv & =B_{1}u\\
M_{1,30}v+M_{1,31}\zeta+M_{1,32}w+M_{1,33}y & =B_{2}u.
\end{align*}
Since the operator matrix $\left(\begin{array}{cc}
M_{1,22} & M_{1,23}\\
M_{1,32} & M_{1,33}
\end{array}\right)\in L(V\oplus Y,V\oplus Y)$ is boundedly invertible by the assumption, we get that 
\[
\left(\begin{array}{c}
w\\
y
\end{array}\right)=\left(\begin{array}{cc}
M_{1,22} & M_{1,23}\\
M_{1,32} & M_{1,33}
\end{array}\right)^{-1}\left(\begin{array}{c}
B_{1}u-(M_{1,20}+C)v-M_{1,21}\zeta\\
B_{2}u-M_{1,30}v-M_{1,31}\zeta
\end{array}\right).
\]
Thus $w$ can be expressed by $v,u$ and $\zeta$. If we plug this
expression for $w$ into equality (\ref{eq:bd}) we obtain a boundary
control equation. Likewise we may assume that the operator matrix
$\left(\begin{array}{cc}
-M_{1,22} & B_{1}\\
-M_{1,32} & B_{2}
\end{array}\right)\in L(V\oplus U,V\oplus Y)$ is boundedly invertible and hence we get that 
\[
\left(\begin{array}{c}
w\\
u
\end{array}\right)=\left(\begin{array}{cc}
-M_{1,22} & B_{1}\\
-M_{1,32} & B_{2}
\end{array}\right)^{-1}\left(\begin{array}{c}
M_{1,23}y+(M_{1,20}+C)v+M_{1,21}\zeta\\
M_{1,33}y+M_{1,30}v+M_{1,31}\zeta
\end{array}\right).
\]
This yields an expression of $w$ in terms of $y,v$ and $\zeta$
and hence (\ref{eq:bd}) becomes a boundary observation equation.\end{rem}
\begin{example}
\label{We-discuss-a}We discuss a possible choice for the observation
space, which will come in handy when we consider the wave equation
with boundary control and observation in the next section. This particular
choice for the control and observation space can be interpreted as
abstract implementation of $L^{2}\left(\Gamma\right)$ of the boundary
$\Gamma$ of the underlying region. To this end, assume that we are
given a continuous linear operator $N:BD(G)\to BD(D)$ satisfying
\[
\left\langle \left(\stackrel{\bullet}{D}N+N^{\ast}\stackrel{\bullet}{G}\right)\phi|\phi\right\rangle _{BD(G)}>0\qquad\left(\phi\in BD(G)\setminus\{0\}\right).
\]
Consider the following sesqui-linear form on $BD(G):$
\[
\langle\cdot|\cdot\rangle_{U}\colon BD(G)\times BD(G)\ni(f,g)\mapsto\frac{1}{2}\langle Nf|\stackrel{\bullet}{G}g\rangle_{BD(D)}+\frac{1}{2}\langle\stackrel{\bullet}{G}f|Ng\rangle_{BD(D)}.
\]
For $f\in BD(G)\setminus\{0\}$, we get 
\[
\frac{1}{2}\langle Nf|\stackrel{\bullet}{G}f\rangle_{BD(G)}+\frac{1}{2}\langle\stackrel{\bullet}{G}f|Nf\rangle_{BD(G)}=\frac{1}{2}\left\langle \left.\left(\stackrel{\bullet}{D}N+N^{\ast}\stackrel{\bullet}{G}\right)f\right|f\right\rangle _{BD(G)}>0.
\]
Hence, $\langle\cdot|\cdot\rangle_{U}$ is an inner product on $BD(G)$.
We denote by $U$ the completion of $BD(G)$ with respect to the norm
induced by $\langle\cdot|\cdot\rangle_{U}$. Then $U$ is a Hilbert
space and 
\begin{align*}
j:BD(G) & \to U\\
f & \mapsto f
\end{align*}
is a dense and continuous embedding. We compute $j^{*}$. Let $f\in BD(G)$
and $g\in BD(G)\subseteq U.$ Then 
\begin{eqnarray*}
\langle j^{\ast}g|f\rangle_{BD(G)} & = & \langle g|jf\rangle_{U}\\
 & = & \frac{1}{2}\langle Ng|\stackrel{\bullet}{G}f\rangle_{BD(D)}+\frac{1}{2}\langle\stackrel{\bullet}{G}g|Nf\rangle_{BD(D)}\\
 & = & \frac{1}{2}\langle Ng|\stackrel{\bullet}{G}f\rangle_{BD(D)}+\frac{1}{2}\langle N^{\ast}\stackrel{\bullet}{G}g|f\rangle_{BD(G)}\\
 & = & \frac{1}{2}\langle Ng|\stackrel{\bullet}{G}f\rangle_{H_{0}(|D|+\i)}+\frac{1}{2}\langle\stackrel{\bullet}{D}Ng|f\rangle_{H_{0}(|G|+\i)}\\
 &  & +\frac{1}{2}\langle N^{\ast}\stackrel{\bullet}{G}g|f\rangle_{H_{0}(|G|+\i)}+\frac{1}{2}\langle\stackrel{\bullet}{G}N^{\ast}\stackrel{\bullet}{G}g|\stackrel{\bullet}{G}f\rangle_{H_{0}(|D|+\i)}\\
 & = & \frac{1}{2}\langle(D-\mathring{D})\pi_{BD(D)}^{*}Ng+(1-\mathring{D}G)\pi_{BD(G)}^{*}N^{\ast}\stackrel{\bullet}{G}g|f\rangle_{H_{0}(|G|+\i)}\\
 & = & \frac{1}{2}\langle(D-\mathring{D}GD)\pi_{BD(D)}^{*}Ng+(1-\mathring{D}G)\pi_{BD(G)}^{*}N^{\ast}\stackrel{\bullet}{G}g|f\rangle_{H_{0}(|G|+\i)}\\
 & = & \frac{1}{2}\langle(1-\mathring{D}G)\pi_{BD(G)}^{*}(\stackrel{\bullet}{D}N+N^{\ast}\stackrel{\bullet}{G})g|f\rangle_{H_{0}(|G|+\i)}\\
 & = & \frac{1}{2}\langle(\stackrel{\bullet}{D}N+N^{\ast}\stackrel{\bullet}{G})g|f\rangle_{H_{0}(|G|+\i)}+\frac{1}{2}\langle\stackrel{\bullet}{G}(\stackrel{\bullet}{D}N+N^{\ast}\stackrel{\bullet}{G})g|\stackrel{\bullet}{G}f\rangle_{H_{0}(|D|+\i)}\\
 & = & \frac{1}{2}\langle(\stackrel{\bullet}{D}N+N^{\ast}\stackrel{\bullet}{G})g|f\rangle_{BD(G)},
\end{eqnarray*}

\end{example}
which gives 
\[
j^{\ast}g=\frac{1}{2}(\stackrel{\bullet}{D}N+N^{\ast}\stackrel{\bullet}{G})g
\]

or 
\[
\stackrel{\bullet}{G}j^{\ast}g=\left(\frac{1}{2}N+\frac{1}{2}\stackrel{\bullet}{G}N^{\ast}\stackrel{\bullet}{G}\right)g.
\]

This yields 
\begin{equation}
(D-\mathring{D})\pi_{BD(D)}^{*}\stackrel{\bullet}{G}j^{\ast}g=\frac{1}{2}(D-\mathring{D})\pi_{BD(D)}^{*}\left(N+\stackrel{\bullet}{G}N^{\ast}\stackrel{\bullet}{G}\right)g.\label{eq:bc-}
\end{equation}

Let us try to interpret this equation in order to underscore that
this can indeed be considered as an equation between classical boundary
traces if the boundary is sufficiently smooth. So, let $\Omega\subseteq\mathbb{R}^{n}$
be open and let $\grad$ be the weak gradient in $L^{2}(\Omega)$
as introduced in Subsection \ref{sub:Boundary-control-and-wave} and
let $\dive$ be the weak divergence from $L^{2}(\Omega)^{n}$ to $L^{2}(\Omega)$.
We denote the boundary of $\Omega$ by $\Gamma$. Assume that $\Gamma\neq\emptyset$
and that any function $f\in D(\grad)$ admits a trace $f|_{\Gamma}\in L^{2}(\Gamma)$
with continuous trace operator. Moreover, assume that there exists
a well-defined unit outward normal $\mathrm{n}:\Gamma\to\mathbb{R}^{n}$
being such that there exists an extension to $\Omega$ in a way that
this extension (denoted by the same name) satisfies $\mathrm{n}\in L_{\infty}(\Omega)^{n}$
with distributional divergence lying in $L^{\infty}(\Omega)$. Then
the operator $\tilde{N}\colon H_{1}(|\grad|+\i)\to H_{1}(|\dive|+\i),f\mapsto\mathrm{n}f$
is well-defined and continuous. For the choices $D=\dive$, $G=\grad$
and $N=\pi_{BD(\dive)}\tilde{N}\pi_{BD(\grad)}^{*}$ in (\ref{eq:bc-})
we can interpret (\ref{eq:bc-}) as the equality of the Neumann trace
of $\stackrel{\bullet}{G}j^{*}g$ and the trace of $g$. Indeed, for
$f,g\in BD(\grad)$ we compute formally with the help of the divergence
theorem\foreignlanguage{english}{
\begin{eqnarray*}
\intop_{\Gamma}\grad j^{\ast}g\cdot\mathrm{n}\; f & = & \frac{1}{2}\intop_{\Gamma}Ng\cdot\mathrm{n}\; f+\frac{1}{2}\langle N^{\ast}\grad g|f\rangle_{H_{0}(|\grad|+\i)}\\
 &  & +\frac{1}{2}\langle\grad N^{\ast}\grad g|\grad f\rangle_{H_{0}(|\dive|+\i)}\\
 & = & \frac{1}{2}\intop_{\Gamma}Ng\cdot\mathrm{n}\; f+\frac{1}{2}\langle N^{\ast}\grad g|f\rangle_{H_{1}(|\grad|+\i)}\\
 & = & \frac{1}{2}\intop_{\Gamma}Ng\cdot\mathrm{n}\; f+\frac{1}{2}\langle g|\dive(Nf)\rangle_{H_{1}(|\grad|+\i)}\\
 & = & \frac{1}{2}\intop_{\Gamma}Ng\cdot\mathrm{n}\; f+\frac{1}{2}\intop_{\Omega}\dive(gNf)\\
 & = & \frac{1}{2}\intop_{\Gamma}gf+\frac{1}{2}\intop_{\Gamma}g(Nf)\cdot\mathrm{n}\\
 & = & \intop_{\Gamma}gf.
\end{eqnarray*}
}

\section{Some Further Applications\label{sec:Applications}}

\subsection{Boundary Control and Observation for Acoustic Waves \label{sub:Boundary-control-and-wave}}

We introduce the operator 
\[
\grad:D(\grad)\subseteq L^{2}(\Omega)\to L^{2}(\Omega)^{n}
\]

as the usual weak gradient in $L^{2}(\Omega)$ for a suitable domain
$\Omega\subseteq\mathbb{R}^{n}.$ We require that the geometric properties
of $\Omega$ are such that $\grad$ is injective and that the range
$\grad[L^{2}(\Omega)]$ is closed%
\footnote{This holds if a Poincare-Wirtinger-type inequality holds, which is
for example the case, if $\Omega$ is connected, bounded in one direction,
satisfies the segment property and possesses infinite Lebesgue-measure.%
} in $L^{2}(\Omega)^{n}$. We choose to use this assumption to avoid
technicalities. If $\grad$ is not injective, one has to proceed similarly
to the way presented in the next section. However, the assumption
on $\grad[L^{2}(\Omega)]\subseteq L^{2}(\Omega)^{n}$ to be closed
is essential. See also the discussion in \cite[Remark 3.1(a)]{TroWau_2012}.
We denote by $\pi_{\grad}:L^{2}(\Omega)^{n}\to\grad[L^{2}(\Omega)]$
the canonical projector induced by the orthogonal decomposition of
$L^{2}(\Omega)^{n}$ with respect to the closed subspace $\grad[L^{2}(\Omega)]$
and consider the operator $\pi_{\grad}\grad:D(\grad)\subseteq L^{2}(\Omega)\to\grad[L^{2}(\Omega)]$.
The negative adjoint of this operator is given by $\interior\dive\pi_{\grad}^{*}:D(\interior\dive)\cap\grad[L^{2}(\Omega)]\subseteq\grad[L^{2}(\Omega)]\to L^{2}(\Omega),$
where $\interior\dive$ is defined as the closure of the divergence
defined on the space of test functions $\interior C_{\infty}(\Omega)^{n}.$
In \cite[Section 7]{Weiss2003} a control system for the wave equation
has been discussed, which has its first order representation in the
system: 
\begin{align}
\left(\partial_{0}\left(\begin{array}{ccc}
1 & \left(\begin{array}{cc}
0 & 0\end{array}\right) & 0\\
\left(\begin{array}{c}
0\\
0
\end{array}\right) & \left(\begin{array}{cc}
1 & 0\\
0 & 0
\end{array}\right) & \left(\begin{array}{c}
0\\
0
\end{array}\right)\\
0 & \left(\begin{array}{cc}
0 & 0\end{array}\right) & 0
\end{array}\right)+\left(\begin{array}{ccc}
0 & \left(\begin{array}{cc}
0 & 0\end{array}\right) & 0\\
\left(\begin{array}{c}
0\\
0
\end{array}\right) & \left(\begin{array}{cc}
0 & 0\\
0 & 1
\end{array}\right) & \left(\begin{array}{c}
0\\
0
\end{array}\right)\\
0 & \left(\begin{array}{cc}
0 & \sqrt{2}\end{array}\right) & 1
\end{array}\right)\right.+\nonumber \\
\left.\left(\begin{array}{ccc}
0 & \left(-\begin{array}{cc}
\interior\dive|_{\grad[L^{2}(\Omega)]} & -C^{\diamond}\end{array}\right) & 0\\
\left(\begin{array}{c}
-\pi_{\grad}\grad\\
C
\end{array}\right) & \left(\begin{array}{cc}
0 & 0\\
0 & 0
\end{array}\right) & \left(\begin{array}{c}
0\\
0
\end{array}\right)\\
0 & \left(\begin{array}{cc}
0 & 0\end{array}\right) & 0
\end{array}\right)\right)\left(\begin{array}{c}
v\\
\left(\begin{array}{c}
\zeta\\
w
\end{array}\right)\\
y
\end{array}\right)\label{eq:TW}\\
=\delta\otimes\left(\begin{array}{c}
z^{(1)}\\
\left(\begin{array}{c}
z^{(0)}\\
0
\end{array}\right)\\
0
\end{array}\right)+\left(\begin{array}{c}
0\\
\left(\begin{array}{c}
0\\
-\sqrt{2}
\end{array}\right)\\
-1
\end{array}\right)u\nonumber 
\end{align}

Using the Hilbert space $U$ from Example \ref{We-discuss-a}, we
define the operator $C$ by%
\footnote{Note that $|\pi_{\grad}\grad|=|\grad|.$%
} 

\begin{align*}
C:H_{1}\left(|\grad|+\i\right) & \to U\\
u & \mapsto-bj\pi_{BD(\grad)}u,
\end{align*}

where $b\in L(U)$. Then we are in the situation of Theorem \ref{thm:adjoint_F}
and hence Corollary \ref{cor:control_well_posed} is applicable. The
state space of equation (\ref{eq:TW}) is given by $H=L^{2}(\Omega)\oplus\grad[L^{2}(\Omega)]\oplus U\oplus U.$
We compute $C^{\diamond}$ with respect to the Gelfand-triplet $H_{1}(|\grad|+\i)\subseteq H_{0}(|\grad|+\i)\subseteq H_{-1}(|\grad|+\i)$.
For $u\in H_{1}(|\grad|+\i),v\in BD(\grad)\subseteq U$, using Example
\ref{As-an-application}, we get that
\begin{align*}
\langle-C^{\diamond}v|u\rangle_{H_{0}(|\grad|+\i)} & =-\langle v|Cu\rangle_{U}\\
 & =\langle v|bj\pi_{BD(\grad)}u\rangle_{U}\\
 & =\langle j^{*}b^{*}v|\pi_{BD(\grad)}u\rangle_{BD(\grad)}\\
 & =\langle\pi_{BD(\grad)}^{\diamond}j^{*}b^{*}v|u\rangle_{H_{0}(|\grad|+\i)}\\
 & =\langle(\pi_{BD(\grad)}^{*}-\interior\dive\pi_{BD(\dive)}^{*}\stackrel{\bullet}{\grad})j^{*}b^{*}v|u\rangle_{H_{0}(|\grad|+\i)}\\
 & =\langle(\pi_{BD(\grad)}^{*}\stackrel{\bullet}{\dive}\stackrel{\bullet}{\grad}-\interior\dive\pi_{BD(\dive)}^{*}\stackrel{\bullet}{\grad})j^{*}b^{*}v|u\rangle_{H_{0}(|\grad|+\i)}\\
 & =\langle(\dive-\interior\dive)\pi_{BD(\dive)}^{*}\stackrel{\bullet}{\grad}j^{*}b^{*}v|u\rangle_{H_{0}(|\grad|+\i)}
\end{align*}

and we read off that $C^{\diamond}v=\left(-(\dive-\interior\dive)\pi_{BD(\dive)}^{*}\stackrel{\bullet}{\grad}j^{*}b^{*}\right)v\in H_{-1}(|\grad|+\i)$
for all $v\in BD(\grad)\subseteq U$. Hence, using (\ref{eq:bd}),
we write the boundary equation as
\[
\pi_{BD(\dive)}\left(\zeta-\interior\dive|_{\grad[L^{2}(\Omega)]}^{-1}\left((\dive-\interior\dive)\pi_{BD(\dive)}^{*}\stackrel{\bullet}{\grad}j^{*}b^{*}\right)w\right)=0.
\]
Since 
\[
\interior\dive|_{\grad[L^{2}(\Omega)]}^{-1}\left((\dive-\interior\dive)\pi_{BD(\dive)}^{*}\stackrel{\bullet}{\grad}j^{*}b^{*}\right)w\in H_{1}(|\dive|+\i),
\]

we get that
\begin{align*}
\pi_{BD(\dive)}\zeta & =\pi_{BD(\dive)}\interior\dive|_{\grad[L^{2}(\Omega)]}^{-1}\left((\dive-\interior\dive)\pi_{BD(\dive)}^{*}\stackrel{\bullet}{\grad}j^{*}b^{*}\right)w\\
 & =-\stackrel{\bullet}{\grad}j^{*}b^{*}w.
\end{align*}

To invoke the boundary control and observation equation we compute
\[
\left(\begin{array}{c}
w\\
y
\end{array}\right)=\left(\begin{array}{cc}
1 & 0\\
\sqrt{2} & 1
\end{array}\right)^{-1}\left(\begin{array}{c}
-\sqrt{2}u-Cv\\
-u
\end{array}\right)=\left(\begin{array}{cc}
1 & 0\\
-\sqrt{2} & 1
\end{array}\right)\left(\begin{array}{c}
-\sqrt{2}u-Cv\\
-u
\end{array}\right)
\]

and 
\[
\left(\begin{array}{c}
w\\
u
\end{array}\right)=-\left(\begin{array}{cc}
1 & \sqrt{2}\\
\sqrt{2} & 1
\end{array}\right)^{-1}\left(\begin{array}{c}
Cv\\
y
\end{array}\right)=\left(\begin{array}{cc}
1 & -\sqrt{2}\\
-\sqrt{2} & 1
\end{array}\right)\left(\begin{array}{c}
Cv\\
y
\end{array}\right).
\]

Thus, we get $w=-\sqrt{2}u-Cv$ and $w=Cv-\sqrt{2}y.$ This yields
\[
\pi_{BD(\dive)}\zeta=\sqrt{2}\stackrel{\bullet}{\grad}j^{*}b^{\ast}u-\stackrel{\bullet}{\grad}j^{*}b^{\ast}bj\pi_{BD(\grad)}v
\]

and 
\[
\pi_{BD(\dive)}\zeta=\stackrel{\bullet}{\grad}j^{*}b^{\ast}bj\pi_{BD(\grad)}v+\sqrt{2}\stackrel{\bullet}{\grad}j^{*}b^{\ast}y.
\]

\begin{rem}
Let us assume that there exists a outward unit normal $\mathrm{n}$
on $\Gamma\coloneqq\overline{\Omega}\setminus\interior\Omega$ such
that there exists a bounded, measurable extension to $\Omega$ with
bounded, measurable distributional divergence. Using the interpretation
from Example \ref{We-discuss-a}, the assumption $b^{*}u,b^{*}Cv,b^{*}y\in BD(\grad)$%
\footnote{In \cite{Weiss2003} these assumptions are formulated with the help
of a certain quotient space $Z_{0}$.%
} and imposing suitable additional requirements on the underlying domain,
we can interpret the latter equations as
\begin{align*}
\mathrm{n}\cdot\zeta & =-b^{*}bv+\sqrt{2}b^{*}u\\
\mathrm{n}\cdot\zeta & =b^{*}bv+\sqrt{2}b^{*}y
\end{align*}
on $\Gamma$ as boundary control and boundary observation equation,
respectively. These correspond to the boundary equations originally
considered in \cite[Section 7]{Weiss2003}. 
\end{rem}
~
\begin{rem}
(a) It is also possible to consider a model, where the type of the
partial differential equations changes over the space, i.e., there
are regions, where the equation is parabolic others where the equation
is hyperbolic and regions where the equation is described best by
elliptic. More precisely, assume the open set $\Omega$ under consideration
can be decomposed into three pairwise disjoint measurable parts $\Omega_{e},$
$\Omega_{p}$, $\Omega_{h}$ such that the evolutionary equation may
be written as
\begin{align*}
 & \left(\partial_{0}\left(\begin{array}{ccc}
\chi_{\Omega_{h}}+\chi_{\Omega_{p}} & \left(\begin{array}{cc}
0 & 0\end{array}\right) & 0\\
\left(\begin{array}{c}
0\\
0
\end{array}\right) & \left(\begin{array}{cc}
\chi_{\Omega_{h}} & 0\\
0 & 0
\end{array}\right) & \left(\begin{array}{c}
0\\
0
\end{array}\right)\\
0 & \left(\begin{array}{cc}
0 & 0\end{array}\right) & 0
\end{array}\right)+\left(\begin{array}{ccc}
\chi_{\Omega_{e}} & \left(\begin{array}{cc}
0 & 0\end{array}\right) & 0\\
\left(\begin{array}{c}
0\\
0
\end{array}\right) & \left(\begin{array}{cc}
\chi_{\Omega_{e}}+\chi_{\Omega_{p}} & 0\\
0 & 1
\end{array}\right) & \left(\begin{array}{c}
0\\
0
\end{array}\right)\\
0 & \left(\begin{array}{cc}
0 & \sqrt{2}\end{array}\right) & 1
\end{array}\right)\right.\\
 & +\left.\left(\begin{array}{ccc}
0 & \left(-\begin{array}{cc}
\interior\dive|_{\grad[L^{2}(\Omega)]} & -C^{\diamond}\end{array}\right) & 0\\
\left(\begin{array}{c}
-\pi_{\grad}\grad\\
C
\end{array}\right) & \left(\begin{array}{cc}
0 & 0\\
0 & 0
\end{array}\right) & \left(\begin{array}{c}
0\\
0
\end{array}\right)\\
0 & \left(\begin{array}{cc}
0 & 0\end{array}\right) & 0
\end{array}\right)\right)\left(\begin{array}{c}
v\\
\left(\begin{array}{c}
\zeta\\
w
\end{array}\right)\\
y
\end{array}\right)\\
 & =\delta\otimes\left(\begin{array}{c}
\left(\chi_{\Omega_{h}}+\chi_{\Omega_{p}}\right)z^{(1)}\\
\left(\begin{array}{c}
\chi_{\Omega_{h}}z^{(0)}\\
0
\end{array}\right)\\
0
\end{array}\right)+\left(\begin{array}{c}
0\\
\left(\begin{array}{c}
0\\
-\sqrt{2}
\end{array}\right)\\
-1
\end{array}\right)u.
\end{align*}
Obviously, the well-posedness condition in Corollary \ref{cor:control_well_posed}
is still satisfied. As it can be verified immediately from the equations
in Remark \ref{control_observation_equation}, the control and observation
equations remain the same. However, we find different types of equations
describing the main physical phenomenon. In particular, on $\Omega_{e}$
we have 
\begin{align*}
 & \left(\left(\begin{array}{ccc}
1 & \left(\begin{array}{cc}
0 & 0\end{array}\right) & 0\\
\left(\begin{array}{c}
0\\
0
\end{array}\right) & \left(\begin{array}{cc}
1 & 0\\
0 & 1
\end{array}\right) & \left(\begin{array}{c}
0\\
0
\end{array}\right)\\
0 & \left(\begin{array}{cc}
0 & \sqrt{2}\end{array}\right) & 1
\end{array}\right)\right.+\left.\left(\begin{array}{ccc}
0 & \left(-\begin{array}{cc}
\interior\dive|_{\grad[L^{2}(\Omega)]} & -C^{\diamond}\end{array}\right) & 0\\
\left(\begin{array}{c}
-\pi_{\grad}\grad\\
C
\end{array}\right) & \left(\begin{array}{cc}
0 & 0\\
0 & 0
\end{array}\right) & \left(\begin{array}{c}
0\\
0
\end{array}\right)\\
0 & \left(\begin{array}{cc}
0 & 0\end{array}\right) & 0
\end{array}\right)\right)\left(\begin{array}{c}
v\\
\left(\begin{array}{c}
\zeta\\
w
\end{array}\right)\\
y
\end{array}\right)\\
 & =\left(\begin{array}{c}
0\\
\left(\begin{array}{c}
0\\
-\sqrt{2}
\end{array}\right)\\
-1
\end{array}\right)u.
\end{align*}
Thus, 
\begin{align*}
v-\begin{array}{cc}
\interior\dive|_{\grad[L^{2}(\Omega)]}\zeta & -C^{\diamond}w\end{array} & =0\\
\zeta-\pi_{\grad}\grad v & =0\\
w+Cv & =-\sqrt{2}u\\
\sqrt{2}w+v & =-u,
\end{align*}
which gives
\[
v-\dive\grad v=0
\]
with the (formal) boundary conditions 
\begin{align*}
\mathrm{n}\cdot\grad v & =-b^{*}bv+\sqrt{2}b^{*}u,\\
\mathrm{n}\cdot\grad v & =b^{*}bv+\sqrt{2}b^{*}y.
\end{align*}
 On $\Omega_{p}$ we get, by similar computations,
\[
\partial_{0}v-\dive\grad v=\delta\otimes z^{(1)}
\]
with the (formal) boundary conditions 
\begin{align*}
\mathrm{n}\cdot\grad v & =-b^{*}bv+\sqrt{2}b^{*}u,\\
\mathrm{n}\cdot\grad v & =b^{*}bv+\sqrt{2}b^{*}y,
\end{align*}
and on $\Omega_{h}$ we get correspondingly
\[
\partial_{0}^{2}v-\dive\grad v=\partial_{0}\delta\otimes z^{(1)}+\delta\otimes z^{(0)}
\]
with the same equations on the boundary.

(b) The last example treats local operators with respect to the spatial
variables. Unless the well-posedness condition in Corollary \ref{cor:control_well_posed}
is not violated, we can also treat integral operators as coefficients.
Indeed, the equation 
\begin{align*}
 & \left(\partial_{0}\left(\begin{array}{ccc}
M_{0,00} & \left(\begin{array}{cc}
M_{0,01} & 0\end{array}\right) & 0\\
\left(\begin{array}{c}
M_{0,10}\\
0
\end{array}\right) & \left(\begin{array}{cc}
M_{0,11} & 0\\
0 & 0
\end{array}\right) & \left(\begin{array}{c}
0\\
0
\end{array}\right)\\
0 & \left(\begin{array}{cc}
0 & 0\end{array}\right) & 0
\end{array}\right)+\left(\begin{array}{ccc}
M_{1,00} & \left(\begin{array}{cc}
M_{1,01} & 0\end{array}\right) & 0\\
\left(\begin{array}{c}
M_{1,10}\\
0
\end{array}\right) & \left(\begin{array}{cc}
M_{1,11} & 0\\
0 & 1
\end{array}\right) & \left(\begin{array}{c}
0\\
0
\end{array}\right)\\
0 & \left(\begin{array}{cc}
0 & \sqrt{2}\end{array}\right) & 1
\end{array}\right)\right.\\
 & +\left.\left(\begin{array}{ccc}
0 & \left(-\begin{array}{cc}
\interior\dive|_{\grad[L^{2}(\Omega)]} & -C^{\diamond}\end{array}\right) & 0\\
\left(\begin{array}{c}
-\pi_{\grad}\grad\\
C
\end{array}\right) & \left(\begin{array}{cc}
0 & 0\\
0 & 0
\end{array}\right) & \left(\begin{array}{c}
0\\
0
\end{array}\right)\\
0 & \left(\begin{array}{cc}
0 & 0\end{array}\right) & 0
\end{array}\right)\right)\left(\begin{array}{c}
v\\
\left(\begin{array}{c}
\zeta\\
w
\end{array}\right)\\
y
\end{array}\right)\\
 & =\delta\otimes\left(\begin{array}{ccc}
M_{0,00} & \left(\begin{array}{cc}
M_{0,01} & 0\end{array}\right) & 0\\
\left(\begin{array}{c}
M_{0,10}\\
0
\end{array}\right) & \left(\begin{array}{cc}
M_{0,11} & 0\\
0 & 0
\end{array}\right) & \left(\begin{array}{c}
0\\
0
\end{array}\right)\\
0 & \left(\begin{array}{cc}
0 & 0\end{array}\right) & 0
\end{array}\right)\left(\begin{array}{c}
z^{(1)}\\
\left(\begin{array}{c}
z^{(0)}\\
0
\end{array}\right)\\
0
\end{array}\right)+\left(\begin{array}{c}
0\\
\left(\begin{array}{c}
0\\
-\sqrt{2}
\end{array}\right)\\
-1
\end{array}\right)u
\end{align*}
leads to the same observation and control equation as in (a), but
the operators $M_{i,jk}$ for $i,j,k\in\{0,1\}$ can be matrices with
variable coefficients or integral operators such as negative roots
of the negative Laplacian. 
\end{rem}

\subsection{Boundary Control for Electromagnetic Waves\label{sub:Maxwell's-Equation-with}}

As a second example we consider a boundary control problem for Maxwell's
system. We shall first introduce the operators involved. Throughout
let $\Omega\subseteq\mathbb{R}^{3}$ be an open domain.
\begin{defn}
We define the operator $\interior\curl$ as the closure of the operator
\begin{align*}
\interior C_{\infty}(\Omega)^{3}\subseteq L^{2}(\Omega)^{3} & \to L^{2}(\Omega)^{3}\\
(\phi_{1},\phi_{2},\phi_{3})^{T} & \mapsto\left(\begin{array}{ccc}
0 & -\partial_{3} & \partial_{2}\\
\partial_{3} & 0 & -\partial_{1}\\
-\partial_{2} & \partial_{1} & 0
\end{array}\right)\left(\begin{array}{c}
\phi_{1}\\
\phi_{2}\\
\phi_{3}
\end{array}\right),
\end{align*}
where $\partial_{i}$ denotes the partial derivative with respect
to the $i$-th coordinate. The operator $\interior\curl$ turns out
to be symmetric and we set $\curl\coloneqq\left(\interior\curl\right)^{\ast}$
and obtain the relation 
\[
\interior\curl\subseteq\curl.
\]

\end{defn}
In \cite{Weck2000} the exact controllability of the following problem
was considered 
\begin{align*}
\partial_{0}\varepsilon E+\curl H & =\delta\otimes E^{(0)},\\
\partial_{0}\mu H-\curl E & =\delta\otimes H^{(0)},
\end{align*}

where the control $u\in BD(\curl)$ prescribes the boundary behaviour
of the tangential component of $H,$ i.e., $\pi_{BD(\curl)}H=u.$
This problem can be dealt with in the following way: We introduce
the function $\tilde{H}\coloneqq H-\pi_{BD(\curl)}^{\ast}u$ and formulate
Maxwell's equations for the pair $(E,\tilde{H})$ as follows 
\begin{align*}
\partial_{0}\varepsilon E+\mathring{\curl}\tilde{H} & =\delta\otimes E^{(0)}-\curl\pi_{BD(\curl)}^{\ast}u,\\
\partial_{0}\mu\tilde{H}-\curl E & =\delta\otimes H^{(0)}-\partial_{0}\mu\pi_{BD(\curl)}^{\ast}u
\end{align*}

or in matrix-form 
\[
\left(\partial_{0}\left(\begin{array}{cc}
\varepsilon & 0\\
0 & \mu
\end{array}\right)+\left(\begin{array}{cc}
0 & \mathring{\curl}\\
-\curl & 0
\end{array}\right)\right)\left(\begin{array}{c}
E\\
\tilde{H}
\end{array}\right)=\delta\otimes\left(\begin{array}{c}
E^{(0)}\\
H^{(0)}
\end{array}\right)-\left(\begin{array}{c}
\curl\pi_{BD(\curl)}^{\ast}\\
\partial_{0}\mu\pi_{BD(\curl)}^{\ast}
\end{array}\right)u.
\]

By our general solution theory (Theorem \ref{Thm: sol_th_ev_syst}
and Proposition \ref{Prop: distributional_time}) this system is well-posed
and we obtain a unique solution $(E,\tilde{H})\in H_{\nu,-1}(\mathbb{R};L^{2}(\Omega)^{3}\oplus L^{2}(\Omega)^{3}).$
Since the time derivative of $u$ occurs as a source term, we obtain
a regularity loss of the solution $(E,\tilde{H}),$ although the system
is locally regularizing. In order to detour this regularity loss,
we may follow the strategy of Subsection \ref{sub:Boundary-Control-and}
and point out, which type of boundary control equations can be treated
in this way. 

In the framework of Subsection \ref{sub:Boundary-Control-and}, we
want $\interior\curl$ to play the role%
\footnote{This implies $\interior G=-\interior\curl$ and $D=\curl.$%
} of $\interior D$ and $-\curl$ that of $G.$ In view of Theorem
\ref{thm:adjoint_F} we have to guarantee that $\curl$ is boundedly
invertible. For this purpose we consider the restriction of the operator
$\curl$ given by 
\begin{align*}
\widetilde{\curl}:D(\curl)\cap N\left(\curl\right)^{\bot}\subseteq N\left(\curl\right)^{\bot} & \to\overline{\curl[L^{2}(\Omega)^{3}]}.
\end{align*}

We require that $\Omega$ has suitable geometric properties such that
$\curl[L^{2}(\Omega)^{3}]$ is closed in order to obtain a boundedly
invertible operator.%
\footnote{For example, domains $\Omega\subseteq\mathbb{R}^{3}$ with conical
points, wedges and cups with a cross section satisfying the segment
property. In \cite{picardea-2001} a large class of such domains is
characterized for which the compactness of the embedding $D(\curl)\cap D(\interior\dive)\hookrightarrow L^{2}(\Omega)^{3}$
holds. This compact embedding result implies the desired properties
for $\widetilde{\curl}.$ %
} An easy computation shows that $\left(\widetilde{\curl}\right)^{\ast}=\interior\curl|_{\curl[L^{2}(\Omega)^{3}]}.$
We decompose the Hilbert space $L^{2}(\Omega)^{3}$ into the following
orthogonal subspaces 
\begin{align*}
L^{2}(\Omega)^{3} & =N\left(\curl\right)\oplus N\left(\curl\right)^{\bot}\\
L^{2}(\Omega)^{3} & =N\left(\interior\curl\right)\oplus N\left(\interior\curl\right)^{\bot}
\end{align*}

and denote by $\pi_{0}:L^{2}(\Omega)^{3}\to N\left(\curl\right),$
$\pi_{1}:L^{2}(\Omega)^{3}\to N\left(\curl\right)^{\bot},$ $\interior\pi_{0}:L^{2}(\Omega)^{3}\to N\left(\interior\curl\right)$
and $\interior\pi_{1}:L^{2}(\Omega)^{3}\to N\left(\interior\curl\right)^{\bot}$
the respective orthogonal projections. Since $\pi_{BD(\curl)}H=\pi_{BD(\curl)}\interior\pi_{1}H$
for each $H\in D(\curl)$ we may write the boundary control problem
in the following way
\begin{align*}
 & \left(\partial_{0}\left(\begin{array}{cccc}
\pi_{1}\varepsilon\pi_{1}^{\ast} & \left(\begin{array}{cc}
0 & 0\end{array}\right) & \pi_{1}\varepsilon\pi_{0}^{\ast} & 0\\
\left(\begin{array}{c}
0\\
0
\end{array}\right) & \left(\begin{array}{cc}
\interior\pi_{1}\mu\interior\pi_{1}^{\ast} & 0\\
0 & 0
\end{array}\right) & \left(\begin{array}{c}
0\\
0
\end{array}\right) & \left(\begin{array}{c}
\interior\pi_{1}\mu\interior\pi_{0}^{\ast}\\
0
\end{array}\right)\\
\pi_{0}\varepsilon\pi_{1}^{\ast} & \left(\begin{array}{cc}
0 & 0\end{array}\right) & \pi_{0}\varepsilon\pi_{0}^{\ast} & 0\\
0 & \left(\begin{array}{cc}
\interior\pi_{0}\mu\interior\pi_{1}^{\ast} & 0\end{array}\right) & 0 & \interior\pi_{0}\mu\interior\pi_{0}^{\ast}
\end{array}\right)\right.\\
 & \left.+\left(\begin{array}{cccc}
0 & \left(\begin{array}{cc}
0 & 0\end{array}\right) & 0 & 0\\
\left(\begin{array}{c}
0\\
M_{1,31}
\end{array}\right) & \left(\begin{array}{cc}
0 & 0\\
M_{1,32} & M_{1,33}
\end{array}\right) & \left(\begin{array}{c}
0\\
M_{1,34}
\end{array}\right) & \left(\begin{array}{c}
0\\
M_{1,35}
\end{array}\right)\\
0 & \left(\begin{array}{cc}
0 & 0\end{array}\right) & 0 & 0\\
0 & \left(\begin{array}{cc}
0 & 0\end{array}\right) & 0 & 0
\end{array}\right)\right.\\
 & \left.+\left(\begin{array}{cccc}
0 & \left(\begin{array}{cc}
\left(\widetilde{\curl}\right)^{\ast} & C^{\diamond}\end{array}\right) & 0 & 0\\
\left(\begin{array}{c}
-\widetilde{\curl}\\
-C
\end{array}\right) & \left(\begin{array}{cc}
0 & 0\\
0 & 0
\end{array}\right) & \left(\begin{array}{c}
0\\
0
\end{array}\right) & \left(\begin{array}{c}
0\\
0
\end{array}\right)\\
0 & \left(\begin{array}{cc}
0 & 0\end{array}\right) & 0 & 0\\
0 & \left(\begin{array}{cc}
0 & 0\end{array}\right) & 0 & 0
\end{array}\right)\right)\left(\begin{array}{c}
\pi_{1}E\\
\left(\begin{array}{c}
\interior\pi_{1}H\\
w
\end{array}\right)\\
\pi_{0}E\\
\interior\pi_{0}H
\end{array}\right)\\
 & =\delta\otimes\left(\begin{array}{c}
\pi_{1}E^{(0)}\\
\left(\begin{array}{c}
\interior\pi_{1}H^{(0)}\\
0
\end{array}\right)\\
\pi_{0}E^{(0)}\\
\interior\pi_{0}H^{(0)}
\end{array}\right)+\left(\begin{array}{c}
0\\
\left(\begin{array}{c}
0\\
Bu
\end{array}\right)\\
0\\
0
\end{array}\right),
\end{align*}

where 
\begin{align*}
C:H_{1}\left(|\widetilde{\curl}|+\i\right) & \to U
\end{align*}

is a bounded linear operator, $U$ an arbitrary Hilbert space and
$B\in L(U)$. The linear operators $M_{1,3i}$ for $i\in\{1,\ldots,5\}$
are bounded in the respective Hilbert spaces and $\Re M_{1,33}$ is
assumed to be strictly positive definite. Since $\widetilde{\curl}$
is boundedly invertible, Theorem \ref{thm:adjoint_F} applies and
Corollary \ref{cor:control_well_posed} yields the well-posedness
of the control problem. The domain of $\left(\begin{array}{cc}
\left(\widetilde{\curl}\right)^{\ast} & C^{\diamond}\end{array}\right)$ reads as 
\begin{align}
 & \left(\begin{array}{c}
\interior\pi_{1}H\\
w
\end{array}\right)\in D\left(\left(\begin{array}{cc}
\left(\widetilde{\curl}\right)^{\ast} & C^{\diamond}\end{array}\right)\right)\nonumber \\
 & \Leftrightarrow\left(\widetilde{\curl}\right)^{\ast}\interior\pi_{1}H+C^{\diamond}w\in H_{0}(|\widetilde{\curl}|+\i)\nonumber \\
 & \Leftrightarrow\left(\widetilde{\curl}\right)^{\ast}\left(\interior\pi_{1}H+\left(\left(\widetilde{\curl}\right)^{\ast}\right)^{-1}C^{\diamond}w\right)\in H_{0}(|\widetilde{\curl}|+\i)\nonumber \\
 & \Leftrightarrow\interior\pi_{1}H+\left(\left(\widetilde{\curl}\right)^{\ast}\right)^{-1}C^{\diamond}w\in H_{1}\left(\left|\left(\widetilde{\curl}\right)^{\ast}\right|+\i\right)\nonumber \\
 & \Leftrightarrow\pi_{BD(\curl)}\left(\interior\pi_{1}H+\left(\left(\widetilde{\curl}\right)^{\ast}\right)^{-1}C^{\diamond}w\right)=0\label{eq:bd_control}
\end{align}

for each $w\in U,H\in L^{2}(\Omega)^{3}.$ By the $3^{\text{{rd}}}$
equation of the above boundary control problem, we get that 
\[
w=-M_{1,33}^{-1}\left((M_{1,31}-C)\pi_{1}E+M_{1,32}\interior\pi_{1}H+M_{1,34}\pi_{0}E+M_{1,35}\interior\pi_{0}H-Bu\right)
\]

and hence (\ref{eq:bd_control}) yields

\begin{align*}
 & \pi_{BD(\curl)}\left(\interior\pi_{1}H-\left(\left(\widetilde{\curl}\right)^{\ast}\right)^{-1}C^{\diamond}M_{1,33}^{-1}\Big((M_{1,31}-C)\pi_{1}E+M_{1,32}\interior\pi_{1}H\right.\\
 & \left.\left.\phantom{\left(\left(\widetilde{\curl}\right)^{\ast}\right)^{-1}}+M_{1,34}\pi_{0}E+M_{1,35}\interior\pi_{0}H-Bu\right)\right)=0
\end{align*}

Although this equation covers a number of possible control equations,
it appears that in this setting the term $(M_{1,31}-C)\pi_{1}E$ cannot
be made to vanish, since we have to assume that $M_{1,31}$ is bounded
on $H_{0}(|\widetilde{\curl}|+\i)$ whereas in general $C$ is not.
This shows that in this setting only boundary control equations containing
terms in $\interior\pi_{1}H$ and $\pi_{1}E$ can be treated without
more intricate adjustments.

\begin{comment}

\end{comment}


\begin{thebibliography}{References}
\bibitem{Arlinski_passive} Y.M. Arlinskii, S. Hassi, and H. De Snoo.
\newblock {Passive systems with a normal main operator and quasi-selfadjoint
systems.}\newblock{ \em Complex Anal. Oper. Theory} 3(1):19--56,
2009.\end{thebibliography}

\begin{thebibliography}{References}
\bibitem{Avrov}D.Z. Arov, M. Kurula, and O.J. Staffans. \newblock{ Canonical
state/signal shift realizations of passive continuous time behaviors.}
\newblock{  \em Complex Anal. Oper. Theory } 5(2): 331--402, 2011.

\bibitem{Ball_Staffans} J.A. Ball and O.J. Staffans. \newblock{ Conservative
state-space realizations of dissipative system behaviors.}\newblock{\em Integral
Equations Oper. Theory} 54(2):151--213, 2006.

\bibitem{Behrndt_bdy_rel} J. Behrndt, S. Hassi, and H. De Snoo. \newblock{
Boundary relations, unitary colligations, and functional models.}
\newblock{\em Complex Anal. Oper. Theory } 3(1), 57--98, 2009.

\bibitem{Behrndt_Kreusler_bdy_rel}J. Behrndt, and H.-C. Kreusler.
\newblock{ Boundary relations and generalized resolvents of symmetric
operators in Krein spaces.} \newblock{\em Integral Equations Oper.
Theory} 59(3): 309--327, 2007.

\bibitem{Curtain1989}R. F. Curtain and G. Weiss. \newblock {Well
posedness of triples of operators (in the sense of linear systems
theory).} \newblock {Control and estimation of distributed parameter
systems, 4th Int. Conf., Vorau/Austria 1988, ISNM 91, 41-59 (1989).}

\bibitem{Engel1998} K.-J. Engel. \newblock {On the characterization
of admissible control- and observation operators.} \newblock {\em Syst. Control Lett.},
34(4):225--227, 1998.

\bibitem{Derkach_bdy_relations} V. Derkach, S. Hassi, M. Malamud,
and H. De Snoo. \newblock {Boundary relations and generalized resolvents
of symmetric operators.}\newblock {\em Russ. J. Math. Phys.} 16(1):17--60,
2009.

\bibitem{Jacob2004}B. Jacob and J. R. Partington. \newblock {Admissibility
of control and observation operators for semigroups: a survey.} \newblock {Ball,
Joseph A. (ed.) et al., Current trends in operator theory and its
applications. Proceedings of the international workshop on operator
theory and its applications (IWOTA), Virginia Tech, Blacksburg, VA,
USA, August 6--9, 2002. Basel: Birkhäuser. Operator Theory: Advances
and Applications 149, 199-221 (2004).}

\bibitem{JacZwartIsem}B.~Jacob and H.~J.~Zwart. \newblock {14th
Internet Seminar: Infinite-dimensional Linear Systems Theory} \newblock
\url{http://www.math.ist.utl.pt/\textasciitilde{}czaja/ISEM/internetseminar201011.pdf}

\bibitem{JacZwart1}B.~Jacob and H.~J.~Zwart. \newblock {Linear
Port-Hamiltonian systems on infinite-dimensional spaces. } \newblock
Operator Theory: Advances and Applications 223. Basel: Birkh\"auser,
2012. 

\bibitem{Kalauch} A.~Kalauch, R.~Picard, S.~Siegmund, S.~Trostorff,
and M.~Waurick. \newblock {A Hilbert Space Perspective on Ordinary
Differential Equations with Memory Term}. \newblock Technical report,
TU Dresden, 2011.

\bibitem{GorZwarMa2006}Y.~Le Gorrec, H.~Zwart, and B.~Maschke.
\newblock {\emph{Dirac structures and boundary control systems associated
with skew-symmetric differential operators.}} \newblock SIAM J.
Control Optim. 44(5): 1864--1892,2006

\bibitem{LasTrigg2000_1}I. Lasiecka and R. Triggiani. \newblock
{ \em Control theory for partial differential equations: continuous
and approximation theories. 1: Abstract parabolic systems}, volume
74 of {\em Encyclopedia of Mathematics and Its Applications.} \newblock
{ Cambridge University Press. Cambridge}, 2000.

\bibitem{LasTrigg2000_2}I. Lasiecka and R. Triggiani. \newblock
{ \em Control theory for partial differential equations: continuous
and approximation theories. 2: Abstract hyperbolic-like systems over
a finite time horizon.}, volume 75 of {\em Encyclopedia of Mathematics
and Its Applications.} \newblock { Cambridge University Press.
Cambridge}, 2000.

\bibitem{Malinen_Staffans_impedance}J. Malinen, and O.J. Staffans
\newblock {Impedance passive and conservative boundary control systems.}
\newblock{\em Complex Anal. Oper. Theory} 1(2):279--300, 2007.

\bibitem{Malinen_Staffans_JDE}J. Malinen, and O.J. Staffans. \newblock {Conservative
boundary control systems.} \newblock{ \em J. Differ. Equations}
231(1): 290--312, 2006.

\bibitem{picardea-2001} R. Picard, N. Weck, and K.-J. Witsch. $\newblock$
{Time-harmonic Maxwell equations in the exterior of perfectly conducting,
irregular obstacles.}{\em {Analysis}} , 21:231--263, 2001. 

\bibitem{Pi2009-1} R.~Picard. \newblock {A Structural Observation
for Linear Material Laws in Classical Mathematical Physics.} \newblock
{\em {Math. Methods Appl. Sci.}}, 32(14):1768--1803, 2009.

\bibitem{pre05760017} R.~Picard. \newblock {On a comprehensive
class of linear material laws in classical mathematical physics.}
\newblock {\em Discrete Contin. Dyn. Syst., Ser. S}, 3(2):339--349,
2010.

\bibitem{PDE_DeGruyter} R.~Picard and D.~F. McGhee. \newblock
{\em Partial Differential Equations: A unified Hilbert Space Approach},
volume~55 of {\em {De Gruyter Expositions in Mathematics}}.
\newblock {De Gruyter. Berlin, New York. 518 p.}, 2011.

\bibitem{PiTroWau_2012} R. Picard, S. Trostorff, and M. Waurick.
\newblock {A note on a class of conservative, well-posed linear systems.}
\newblock \newblock{\em 8th ISAAC Congress, Session on Evolution
Partial Differential Equations, M. Reissig, M. Ruzhansky (eds.)}
Springer Proceedings in Mathematics \& Statistics (PROMS). To appear.

\bibitem{Salamon1987}D. Salamon. \newblock {Infinite dimensional
linear systems with unbounded control and observation: A functional
analytic approach.} \newblock {\em Trans. Am. Math. Soc.}, 300:383--431,
1987.

\bibitem{Salamon1989}D. Salamon. \newblock {Realization theory in
Hilbert space.} \newblock {\em Math. Syst. Theory}, 21(3):147--164,
1989.

\bibitem{Staffans_passive}O.J. Staffans \newblock{ Passive and conservative
continuous-time impedance and scattering systems. I: Well-posed systems.}\newblock{\em Math.
Control Signals Syst.}, 15(4):291--315, 2002.

\bibitem{Staffasn_Jen}O.J. Staffans. \newblock{$J$-energy preserving
well-posed linear systems.}  \newblock{ \em Int. J. Appl. Math.
Comput. Sci. } 11(6): 1361--1378, 2001.

\bibitem{BDY_systems}C. Schubert, C. Seifert, J. Voigt, and M. Waurick.
\newblock{ Boundary systems and self-adjoint operators on infinite
metric graphs.} Submitted.

\bibitem{TroWau_2012} S. Trostorff and M. Waurick. \newblock A note
on elliptic type boundary value problems with maximal monotone relations.
\newblock  Technical Report, TU Dresden, 2012.

\bibitem{SWIPs}M. Waurick, and M. Kaliske. \newblock{ On the well-posedness
of evolutionary equations on infinite graphs.} \newblock{W. Arendt
(ed.) et al., \emph{Spectral theory, mathematical system theory, evolution
equations, differential and difference equations. Selected papers
of 21st international workshop on operator theory and applications,
IWOTA10, Berlin, Germany, July 12--16, 2010. Basel},} \newblock{ Birkh\"auser.
Operator Theory: Advances and Applications} 221: 653--666, 2012.

\bibitem{Weck2000}N. Weck. \newblock Exact Boundary Controllability
of a Maxwell Problem. \newblock {\em SIAM J. Control Optim.}, 38(3):736--750,
2000.

\bibitem{Weidmann}J. Weidmann. \newblock{\em  Linear Operators in
Hilbert Spaces.}\newblock{ Springer, New York, 1980.}

\bibitem{Weiss1989-2} G. Weiss. \newblock {Admissibility of unbounded
control operators.} \newblock {\em SIAM J. Control Optim.}, 27(3):527--545,
1989.

\bibitem{Weiss1989} G. Weiss. \newblock {The representation of
regular linear systems on Hilbert spaces.} \newblock {Control and
estimation of distributed parameter systems, 4th Int. Conf., Vorau/Austria
1988, ISNM 91, 401--416}, 1989.

\bibitem{WeissStaffTucs}G. Weiss, O.J. Staffans, and M. Tucsnak.
\newblock {Well-posed linear systems -- a survey with emphasis on
conservative systems.} \newblock{\em Int. J. Appl. Math. Comput.
Sci.} 11(1):7--33, 2001.

\bibitem{Weiss2003} G. Weiss and M. Tucsnak. \newblock {How to
get a conservative well-posed linear system out of thin air. I: Well-posedness
and energy balance.} \newblock {\em ESAIM, Control Optim. Calc.
Var.}, 9:247--274, 2003.

\bibitem{ZwartGor2010}H.~Zwart, Y.~Le Gorrec, B.~Maschke, and
J.~Villegas. \newblock{ Well-posedness and regularity of hyperbolic
boundary control systems on a one-dimensional spatial domain.} \newblock{\em 
ESAIM, Control Optim. Calc. Var.} 16(4): 1077--1093, 2010.

\end{thebibliography}
\end{document}